\documentclass[10pt]{amsart}
\usepackage[latin1]{inputenc}
\usepackage[T1]{fontenc}
\usepackage{indentfirst}
\usepackage{textcomp}
\usepackage[english]{babel}
\usepackage{amsmath,amsfonts,amssymb,amsopn,amscd,amsthm}
\usepackage{mathbbol}
\usepackage[enableskew]{youngtab}	
\usepackage{mathrsfs}
\usepackage{graphicx}
\usepackage[dvipsnames]{xcolor}
\usepackage{pstricks,pst-plot}
\usepackage{xy}  \xyoption{all}
\usepackage{pb-diagram,pb-xy}
\usepackage{enumerate}

\begin{document}

\author{Pierre-Lo\"ic M\'eliot}
\title{Asymptotics of the Gelfand models of the symmetric groups}
\address{The Gaspard--Monge Institut of electronic and computer science,
University of Marne-La-Vall\'ee Paris-Est,
77454 Marne-la-Vall\'ee Cedex 2, France}
\email{meliot@phare.normalesup.org}

\newcommand{\Z}{\mathbb{Z}}     
\newcommand{\N}{\mathbb{N}}    
\newcommand{\R}{\mathbb{R}}  
\newcommand{\Q}{\mathbb{Q}}    
\newcommand{\C}{\mathbb{C}}    
\newcommand{\I}{\mathrm{i}}
\newcommand{\E}{\mathrm{e}}
\newcommand{\sym}{\mathfrak{S}}
\newcommand{\Comp}{\mathfrak{C}}   
\newcommand{\Part}{\mathfrak{P}}      
\newcommand{\proba}{\mathbb{P}}     
\newcommand{\esper}{\mathbb{E}}     
\newcommand{\card}{\mathrm{card}}  
\newcommand{\obs}{\mathscr{O}}
\newcommand{\GL}{\mathrm{GL}}
\newcommand{\Gel}{\mathbb{G}}
\newcommand{\For}{\mathbb{F}}
\newcommand{\wt}{\mathrm{wt}}
\newcommand{\eps}{\varepsilon}
\newcommand{\tilh}{\widetilde{h}}
\newcommand{\tilp}{\widetilde{p}}
\newcommand{\tilq}{\widetilde{q}}
\newcommand{\tr}{\mathrm{tr}}
\newcommand{\id}{\mathrm{id}}
\newcommand{\lle}{\left[\!\left[}              
\newcommand{\rre}{\right]\!\right]}    
\newcommand{\scal}[2]{\left\langle #1\vphantom{#2}\,\right |\left.#2 \vphantom{#1}\right\rangle}   
\newcommand{\comment}[1]{}
\newtheorem{theorem}{Theorem}
\newtheorem{definition}[theorem]{Definition}
\newtheorem{proposition}[theorem]{Proposition}
\newtheorem{lemma}[theorem]{Lemma}
\newtheorem{corollary}[theorem]{Corollary}
\theoremstyle{remark}
\newtheorem*{example}{Example}
\newtheorem*{examples}{Examples}
\newcommand{\figcap}[2]{\begin{figure}[ht] \begin{center} 
							{\footnotesize{#1}} 
							\caption{#2} \end{center} \end{figure}}

\begin{abstract}
If a partition $\lambda$ of size $n$ is chosen randomly according to the Plancherel measure $\proba_{n}[\lambda]=(\dim \lambda)^{2}/n!$, then as $n$ goes to infinity, the rescaled shape of $\lambda$ is with high probability very close to a non-random continuous curve $\Omega$ known as the Logan-Shepp-Kerov-Vershik curve (\cite{LS77}, \cite{KV77}). Moreover, the rescaled deviation of $\lambda$ from this limit shape can be described by an explicit generalized gaussian process (see \cite{Ker93}). In this paper, we investigate the analoguous problem when $\lambda$ is chosen with probability proportional to $\dim \lambda$ instead of $(\dim \lambda)^{2}$. We shall use very general arguments due to Ivanov and Olshanski for the first and second order asymptotics (\emph{cf.} \cite{IO02}); these arguments amount essentially to a method of moments in a noncommutative setting. The first order asymptotics of the Gelfand measures turns out to be the same as for the Plancherel measure; on the contrary, the fluctuations are different (and bigger), although they involve the same generalized gaussian process. Many of our computations relie on the enumeration of involutions and square roots in $\sym_{n}$.
\end{abstract}
\maketitle

\section{Gelfand models and associated ensembles of random partitions}

\subsection{Plancherel measure of a complex linear representation of a finite group}
Consider \emph{any} finite group $G$, and \emph{any} (complex, linear, finite-dimensional) representation $V$. It is well-known that $V$ can be written in a unique way as a direct sum of irreducible representations:
$$V=\bigoplus_{\lambda \in \widehat{G}} n_{\lambda}\,V^{\lambda}$$
where $\widehat{G}$ is the set of isomorphism classes of irreducible representations of $G$, and the $n_{\lambda}$'s are non-negative integers. The \textbf{Plancherel measure} of the representation $V$ is the probability measure on $\widehat{G}$ corresponding to this decomposition:
$$\proba_{V}[\lambda \in \widehat{G}]=\frac{n_{\lambda}\,\dim V^{\lambda}}{\dim V}$$
In the special case when $V=\C G$ is the regular representation, $n_{\lambda}=\dim V^{\lambda}$, and $\proba_{\C G}$ is simply called the Plancherel measure of the group. That said, the asymptotic representation theory of a family of groups $(G_{n})_{n \in \N}$ is the set of results related to the following problem: given a ``natural'' family of representations $(V_{n})_{n\in \N}$, what is the asymptotic distribution of $\lambda \in \widehat{G}_{n}$ under the Plancherel measure $\proba_{V_{n}}$? Is it possible to enounce a law of large numbers, and a central limit theorem in this setting?\bigskip

\subsection{Classical ensembles of random partitions coming from representation theory}
When $G=\sym_{n}$ is the symmetric group of order $n$, the irreducible representations are labelled by partitions of size $n$, \emph{i.e.}, non-increasing sequences $\lambda=(\lambda_{1},\ldots,\lambda_{r})$ of positive integers that sum up to $n$. We denote by $\Part_{n}$ the set of partitions of size $n$; it follows from the previous paragraph that any representation of $\sym_{n}$ yields a probability measure on $\Part_{n}$. 
\figcap{$$ \yng(1,3,4,4,6,7)$$\vspace{-5mm}}{The Young diagram of the partition $\lambda=(7,6,4,4,3,1)$ of size $n=25$.\label{youngdiagram}}
Partitions are usually represented by their Young diagrams, see figure \ref{youngdiagram}; \clearpage

\noindent hence, a representation of $\sym_{n}$ corresponds to a random ensemble of such planar objects. Let us describe briefly some ``classical'' ensembles of random partitions constructed in this way:\vspace{2mm}
\begin{enumerate}
\item \textbf{Plancherel measures}. When $G=\sym_{n}$ and $V=\C\sym_{n}$, the probability law $\proba_{V}=\proba_{n}$ is the Plancherel measure of the symmetric group $\sym_{n}$, and its asymptotics have been extensively studied (see in particular \cite{IO02}), in connection with Ulam's problem of longest increasing subsequences in random permutations (\cite{BDJ99}, \cite{BDJ00}), random matrix theory (\cite{BOO00}, \cite{Oko00}), free probability (\cite{Bia98}), random tilings and random surfaces (\cite{OR03}), hydrodynamic partial differential equations (\cite{Ker99}) and Gromov-Witten theory (\cite{Oko03}). The \textbf{Logan-Shepp-Kerov-Vershik law of large numbers} ensures that if the diagram of a random partition of size $n$ is rotated $45^{o}$ and drawn with square boxes of area $2/n$, then as $n$ goes to infinity, the upper boundary $s\mapsto \lambda^{*}(s)$ of this shape converges in probability towards the continuous curve
$$\Omega(s)=\begin{cases}
\frac{2}{\pi}\left(s\,\arcsin \frac{s}{2}+\sqrt{4-s^{2}}\right)&\text{if }|s|<2,\\
|s|&\text{if }|s|\geq2.
\end{cases}$$
see figure \ref{lskv}. Moreover, the scaled deviation $\Delta_{n,\proba}(s)=\sqrt{n}(\lambda^{*}(s)-\Omega(s))$ converges in law in $[-2,2]$ towards the generalized gaussian process
$$\Delta_{\infty,\proba}(s)=\Delta_{\infty,\proba}(2\cos\theta)=\frac{2}{\pi}\sum_{k=2}^{\infty}\frac{\zeta_{k}}{\sqrt{k}}\,\sin k\theta,$$
where the $\zeta_{k}$'s are independent normal variables of variance $1$, and the infinite sum makes sense as a distribution. This latter result is \textbf{Kerov's central limit theorem}.

\figcap{\includegraphics{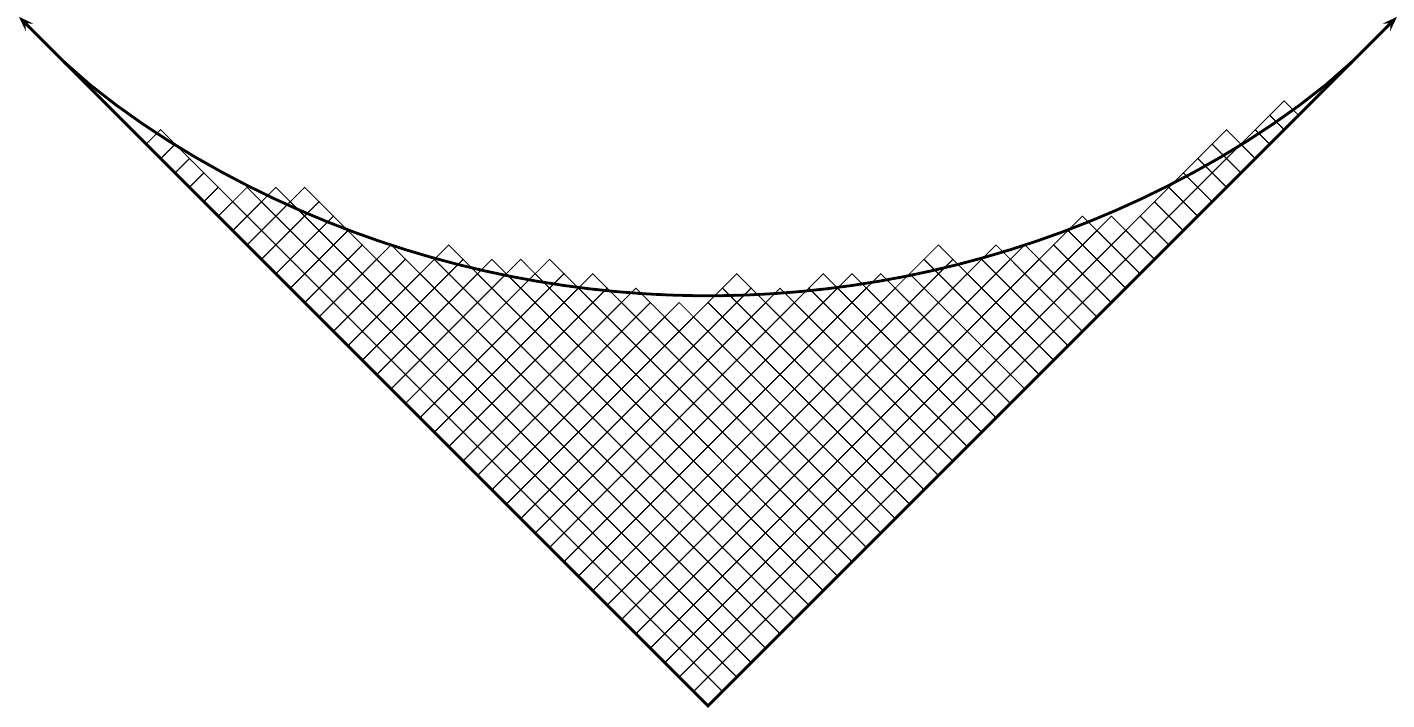}}{
Young diagram of size $n=500$ picked randomly according to the Plancherel measure of the symmetric group.\label{lskv}}

\item \textbf{Schur-Weyl measures}. When $G=\sym_{n}$ acts on $V=(\C^{N})^{\otimes n}$ by permutation of the letters in the simple tensors, one obtains the so-called Schur-Weyl measure $SW_{n,N}$. Notice that by Schur-Weyl duality, one may also consider this measure as a measure on irreducible polynomial representations of $\GL(N,\C)$; in particular, $SW_{n,N}$ charges only the partitions with less than $N$ parts. Let us suppose that $n^{\alpha}\simeq cN$ with $c>0$. Then, one has three asymptotic regimes according to the value of $\alpha$. If $\alpha>1/2$, the random partitions have the same asymptotics as in the Plancherel case (LSKV law of large numbers and Kerov's central limit theorem). When $\alpha=1/2$, the first order asymptotics have been studied by P. Biane in \cite{Bia01}. The limit shapes $\Omega_{c}$ are deformations of $\Omega$, see figure \ref{schurweyl}; the case $c=0$ corresponds to the LSKV curve. The support of $\Omega_{c}(s)-|s|$ is $[c-2,c+2]$ when $c<1$, and $[-1/c,c+2]$ when $c\geq 1$. More recently, we found out that Kerov's central limit theorem also holds in the case of Schur-Weyl measures of parameter $\alpha=1/2$ in the interval $[c-2,c+2]$, see \cite{Mel10b}.
\figcap{\includegraphics{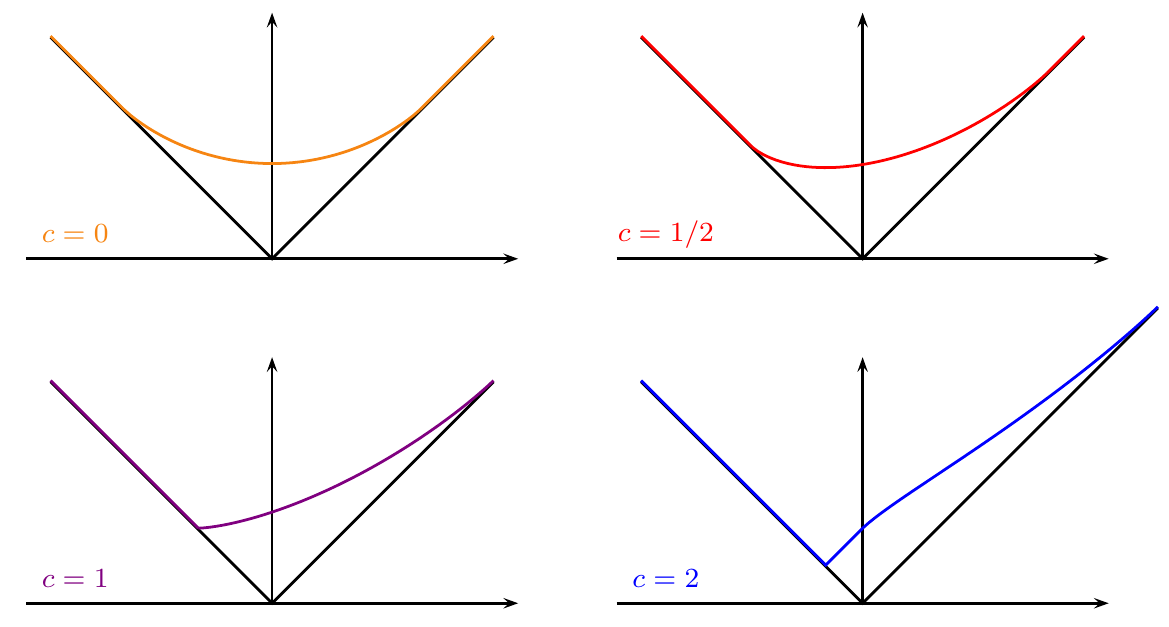}}{Limit shapes $\Omega_{c}$ of the rescaled Young diagrams under the Schur-Weyl measures of parameter $\alpha\geq 1/2$.\label{schurweyl}}
Finally, when $\alpha<1/2$, the lines of the random partitions have typical order of magnitude $n^{1-\alpha}$, and the columns of the random partitions have typical order of magnitude $n^{\alpha}$; and when rescaled according to these orders, the shapes of partitions are concentrated around a rectangle of edges $c$ and $c^{-1}$, see figure \ref{rectangle}.
\figcap{\includegraphics{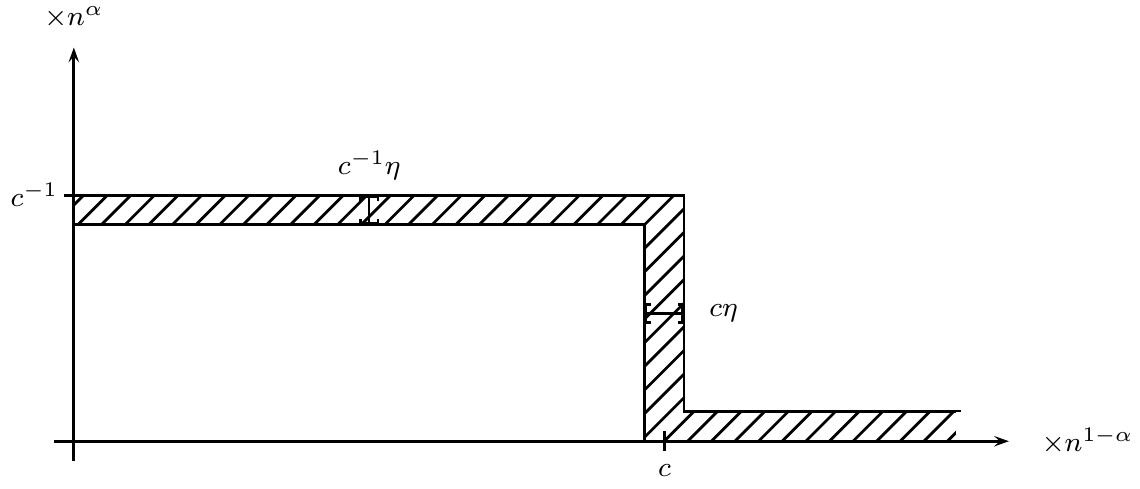}}{Limit shapes of the rescaled Young diagrams under the Schur-Weyl measures of parameter $\alpha<1/2$. When $n$ goes to infinity, with high probability, the rescaled diagram is in the hatched area.\label{rectangle}}

\item \textbf{q-Plancherel measures}. When $G=\GL(n,\For_{q})$ and $V=\C[\GL(n,\For_{q})/\mathrm{B}(n,\For_{q})]=\C[\mathcal{F}(n,\For_{q})]$ is the module of functions on the flag variety of $(\For_{q})^{n}$, one obtains the $q$-Plancherel measure $M_{n,q}$, that has been studied in \cite{FM10} and \cite{Mel10a}. By Iwahori theory, one may also consider $M_{n,q}$ as a measure on irreducible representations of the Hecke algebra $\mathscr{H}_{q}(\sym_{n})$, that is to say, as a measure on partitions. Then, the $i$-th column of a random partition under the law $M_{n,q}$ is asymptotically equal to $(q-1)q^{-i}\,n$ (see figure \ref{qyoung}), and the deviations from these expected values are described by an explicit gaussian process.
\figcap{\includegraphics{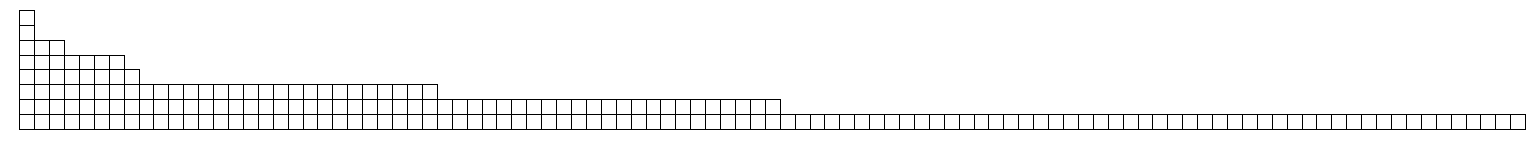}}{Conjugate of a Young diagram of size $n=200$ picked randomly according to the $q$-Plancherel measure of parameter $q=2$.\label{qyoung}}
\end{enumerate}
All these asymptotic results can be obtained by using a noncommutative method of moments, that is essentially due to Kerov and Olshanski; we shall recall this in section \ref{obs}. The purpose of this paper is to give yet another example of ensembles of random partitions that can be studied using these elementary methods. These ensembles come from the so-called \textbf{Gelfand models} of the symmetric groups, and from a certain viewpoint, they are the simplest ensembles that one can think of in the setting of representation theory, because they correspond to the case when all $n_{\lambda}$'s are equal to $1$.\bigskip\bigskip

\subsection{The Gelfand model of the symmetric group and its trace} If $G$ is a finite group, a Gelfand model for $G$ is a representation $M$ whose decomposition in irreducibles is $M=\bigoplus_{\lambda \in \widehat{G}} V_{\lambda}$; in other words, any irreducible representation occurs with multiplicity $1$. In the case of the symmetric group, a combinatorial description of such a representation has been proposed by Adin, Postnikov and Roichman in \cite{APR07}.  In the following, $\varsigma^{\lambda}$ denotes the \emph{non normalized} irreducible character of type $\lambda\in \Part_{n}$. We also denote by $\mathfrak{I}_{n}$ the set of \textbf{involutions} in $\sym_{n}$, and by $\C\mathfrak{I}_{n}$ the vector space of formal sums of involutions; its basis elements shall be denoted by $B_{v}$ if $v \in \mathfrak{I}_{n}$. Recall that the symmetric group $\sym_{n}$ is a Coxeter group with generators $s_{i}=(i,i+1)$, $i \in \lle1,n\rre$. The \textbf{descents} of a permutation $\sigma$ are the indices $i$ such that $\ell(\sigma \,s_{i})<\ell(\sigma)$; we denote by $D(\sigma)$ their set. Then, one can define a representation $\rho : \sym_{n}\mapsto \GL(\C\mathfrak{I}_{n})$ by:
$$s_{i}\cdot B_{v} = \begin{cases}
- B_{v}&\text{if }s_{i}vs_{i}=v \text{ and }i\in D(v),  \\
B_{s_{i}vs_{i}}&\text{otherwise}. 
\end{cases}$$
It is a theorem that $(\rho,\C\mathfrak{I}_{n})$ yields a Gelfand model for $\sym_{n}$. It actually comes from the following property, see Theorem 3.1 in \cite{APR07}:
\begin{proposition}[Trace of the Gelfand model of a $\R$-split group]\label{tracegelfand}
Let $G$ be a finite group such that $\Z G$ is split over $\R$. For any $g \in G$,
$$\sum_{\lambda \in \widehat{G}} \varsigma^{\lambda}(g)=\card\{h \in G\,\,|\,\,h^{2}=g\}.$$
\end{proposition} \bigskip

This proposition applies to $\sym_{n}$, because any Coxeter group is defined over $\R$ (and even $\Q$). On the other hand, one can show that the aforementioned representation $\rho$ has indeed this character; as a consequence, $(\rho,\C\mathfrak{I}_{n})$ is indeed a Gelfand model of $\sym_{n}$. In particular, one sees that:
$$\sum_{\lambda \in \Part_{n}} \dim \lambda = \card\{\sigma \in \sym_{n}\,\,|\,\,\sigma^{2}=\id\}=\card \,\mathfrak{I}_{n}$$
This identity can be recovered if one notices that $\sigma \in \sym_{n}$ is an involution if and only if the two standard tableaux $P(\sigma)$ and $Q(\sigma)$ obtained from $\sigma$ by the RSK correspondence (\emph{cf.} \cite{Ful97}) are the same. Thus\footnote{The Gelfand model $(\rho,\C\mathfrak{I}_{n})$ of $\sym_{n}$ is well-behaved with respect to the RSK correspondence. Hence, for any involution $\sigma$ with conjugacy class $C$, the submodule $M(C)\subset \C\mathfrak{I}_{n}$ generated by $B_\sigma$ has for decomposition:
$$M(C)=\bigoplus_{\substack{\lambda \text{ shape of the standard } \\ \text{tableaux }P(\tau) ,Q(\tau)\text{ with } \tau \in C}}\!\!\!\!\!\!\!\! V^{\lambda}$$},
$$\card\,\mathfrak{I}_{n}=\sum_{\lambda\in \Part_{n}} \card\,\mathrm{Std}(\lambda)=\sum_{\lambda \in \Part_{n}}\dim \lambda.$$
We shall denote by $I_{n}$ the number of involutions of size $n$. A permutation $\sigma \in \sym_{n}$ is an involution if and only if its cycle type is a partition of type $1^{n-2k}\,2^{k}$; so,
$$I_{n}=\sum_{k=0}^{\lfloor \frac{n}{2}\rfloor} \card \,C_{1^{n-2k}\,2^{k}}=\sum_{k=0}^{\lfloor \frac{n}{2}\rfloor} \frac{n!}{k!\,n-2k!\,2^{k}}.$$
\begin{definition}[Gelfand measure]
The Gelfand measure of the symmetric group $\sym_{n}$ is the probability measure $\Gel_{n}$ on $\Part_{n}$ associated to the Gelfand model, that is to say that
$$\Gel_{n}[\lambda]=\frac{\dim\lambda}{\sum_{\mu \in \Part_{n}} \dim \mu}=\frac{\dim\lambda}{I_{n}}.$$
\end{definition}\bigskip

\noindent This Gelfand measure is also the push-forward of the uniform measure on $\mathfrak{I}_{n}$ by the RSK correspondence, and it is a particular case of $\beta$-Plancherel measures, see \cite{BR01} and the last section.  Let us denote by $\chi^{\lambda}$ the \emph{normalized} irreducible character of type $\lambda$; hence, $\chi^{\lambda}(\sigma)=\varsigma^{\lambda}(\sigma)/\dim \lambda$. Then, if $\lambda$ is picked randomly according to the Gelfand measure $\Gel_{n}$, one has:
$$\Gel_{n}[\chi^{\lambda}(\sigma)]=\frac{1}{I_{n}}\,\sum_{\lambda \in \Part_{n}} (\dim \lambda) \,\chi^{\lambda}(\sigma)=\frac{1}{I_{n}}\,\sum_{\lambda \in \Part_{n}} \varsigma^{\lambda}(\sigma)=\frac{\tr\rho(\sigma)}{\tr\rho(\id)}=\frac{\card\{\tau \in \sym_{n}\,\,|\,\,\tau^{2}=\sigma\}}{\card\{\tau \in \sym_{n}\,\,|\,\,\tau^{2}=\id\}}$$
This simple identity will allow the asymptotic study of the Gelfand measures.
\bigskip\bigskip

\section{Bases and filtrations of the algebra of observables of Young diagrams}\label{obs}

As already pointed out in \cite{IO02}, \cite{FM10} or \cite{Mel10b}, a very powerful tool for the asymptotic study of random partitions arising from representation theory is the algebra $\obs$ of \textbf{observables of diagrams}, also known as polynomial functions on diagrams (\emph{cf.} \cite{KO94}). In this section, we recall the combinatorics of three different bases of $\obs$. As in  the papers \cite{FM10}, \cite{Mel10a} and \cite{Mel10b}, we shall also use \'Sniady's theory of cumulants of observables (\emph{cf. \cite{Sni06}}), but this will be recalled in \S\ref{asymptoticindependence}. Using this theory of observables of diagrams, we will be able to prove the asymptotic gaussian behaviour of the rescaled deviations of the observables under the Gelfand measures; but in contrast to the cases of Plancherel and Schur-Weyl measures, these rescaled deviations will be non centered.\bigskip

\subsection{Continuous Young diagrams and renormalizations}\label{n}
If $\lambda$ is an integer partition of size $n$, we rotate its Young diagram by 45 degrees and we consider the upper boundary of the drawing as a continuous function $s \mapsto \lambda(s)$ that is equal to $|s|$ if $s \leq -\ell(\lambda)$ or if $s \geq \lambda_{1}$, see figure \ref{continuous}.
\figcap{\includegraphics{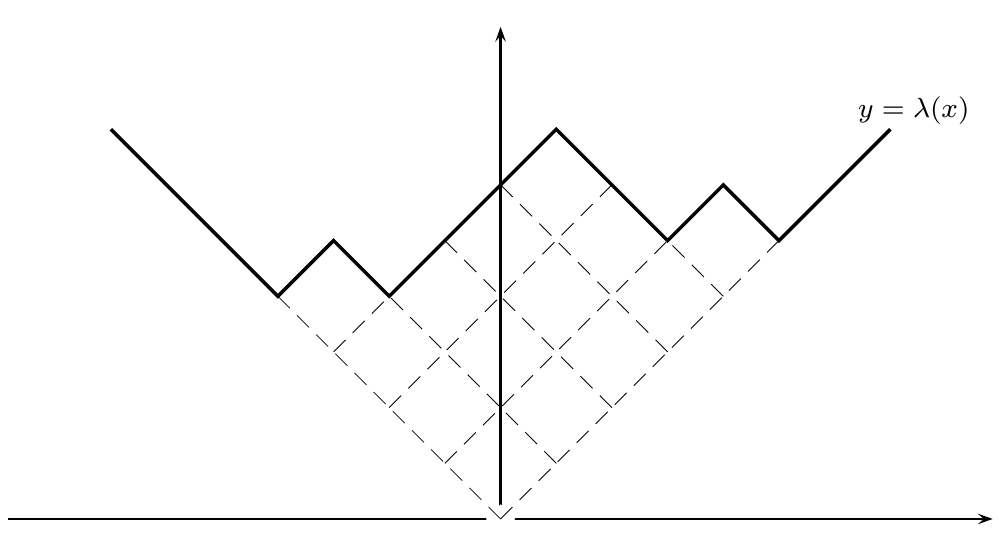}}{Function $s\mapsto \lambda(s)$ associated to the Young diagram $\lambda=(5,4,4,1)$.\label{continuous}}
\clearpage

More generally, we call \textbf{continuous Young diagram} any continuous function $\omega : \R \to \R$ that is Lipschitz with constant $1$, and equal to $|s|$ for $|s|$ big enough. The \textbf{area} of a continuous Young diagram is defined by $\mathcal{A}(\omega)=\int_{\R}(\omega(s)-|s|)\,ds$, and the \textbf{support} of $\omega$ is the support of the function $\omega(s)-|s|$. So for instance, the LSKV curve is a continuous Young diagram with area $2$ and support $[-2,2]$. These notions allow to consider \textbf{renormalized} Young diagrams. So, if $\omega$ is a (continuous) Young diagram and if $t$ is a positive real number, we shall denote by $\omega^{t}$ the renormalized continuous Young diagram
$$\omega^{t}(s)=\frac{\omega(\sqrt{t}s)}{\sqrt{t}}.$$
If $\omega$ had area $\mathcal{A}$ and support $[a,b]$, then $\omega^{t}$ has area $\mathcal{A}/t$ and support $[a/\sqrt{t},b/\sqrt{t}]$. In particular, if $\lambda$ is an integer partition of size $n$, we shall denote by $\lambda^{*}$ the continuous Young diagram $\lambda^{n}$; it always has for area $2$.\bigskip\bigskip

\subsection{Interlacing moments, central characters and free cumulants}
Any partition $\lambda$ is characterized by the two interlacing sequences
$x_{1}<y_{1}<x_{2}<y_{2}<\cdots<x_{s-1}<y_{s-1}<x_{s}$
of the local minima and the local maxima of $s \mapsto \lambda(s)$. In this setting, the $k$-th \textbf{interlacing moment} of $\lambda$ is defined by $$p_{k}(\lambda)=\sum_{i=1}^{s}(x_{i})^{k}-\sum_{i=1}^{s-1}(y_{i})^{k}.$$
Notice that $p_{1}(\lambda)$ is always zero, whereas $p_{2}(\lambda)=\mathcal{A}(\lambda)=2|\lambda|$. More generally, if $\lambda$ is a \emph{continuous} Young diagram, we define its $k$-th interlacing moment by
$$p_{k}(\omega)=\int_{\R} \sigma_{\omega}''(s)\,s^{k}\,ds,$$
where $\sigma_{\omega}=\frac{\omega(s)-|s|}{2}$ is the \textbf{charge} of the continuous Young diagram. The \textbf{generating function} of the continuous Young diagram $\lambda$ is defined by:
$$G_{\omega}(z)=\frac{1}{z}\exp\left(\sum_{k=1}^{\infty}\frac{p_{k}(\omega)}{k}\,z^{-k}\right)=\frac{1}{z}\,\exp\left(-\int_{\R}\frac{\sigma'(s)}{z-s}\,ds\right)$$
For a true Young diagram, $G_{\lambda}(z)=\frac{\prod_{i=1}^{v-1}z-y_{i}}{\prod_{i=1}^{v}z-x_{i}}$. An \textbf{observable of diagrams} is a (complex) linear combination of products $p_{\mu}=p_{\mu_{1}}\,p_{\mu_{2}}\cdots p_{\mu_{r}}$ of interlacing moments of continuous Young diagrams. These functions form a commutative algebra $\obs$ that is graded by $\mathrm{wt}(p_{k\geq 2})=k$, and $\mathrm{wt}(p_{1})=0$. Another basis of $\obs$ is provided by the free cumulants, that are defined in the following way. For any continuous diagram $\omega$, $G_{\omega}(z)\simeq 1/z$ when $z$ goes to infinity, so $G_{\omega}(z)$ admits a formal inverse $R_{\omega}(z)$ in a vicinity of $\infty$, and
$$R_{\omega}(z)=\frac{1}{z}\left(1+\sum_{k=1}^{\infty}R_{k}(\omega)\,z^{k}\right).$$
Notice that $R_{1}(\omega)$ is always $0$. The coefficients $R_{k\geq 2}(\omega)$ are called \textbf{free cumulants} of the continuous Young diagram $\omega$, and by using Lagrange inversion formula, one can show that these functions form an algebraic basis of $\obs$. Moreover, $\mathrm{wt}(R_{k})=k$ for any $k\geq 2$, and $R_{k}$ is an homogeneous observable. For more details on the combinatorics of free cumulants of diagrams, we refer to \cite{Bia98} and \cite{Bia03}.\bigskip\bigskip

A remarkable fact is that $\obs$ is also linearly generated by rescaled versions of the irreducible characters of the symmetric groups (\cite{IO02}). More precisely, if $\mu$ is a partition of size $k$, we introduce the \textbf{central character}
$$\varSigma_{\mu}(\lambda)=\begin{cases}
n^{\downarrow k}\,\chi^{\lambda}(\mu \sqcup 1^{n-k})&\text{if }|\lambda|=n\geq k,\\
0&\text{if }|\lambda|=n<k,
\end{cases}$$
where the falling factorial is defined by $n^{\downarrow k}=n(n-1)\cdots(n-k+1)$. It can be shown that $\varSigma_{\mu}$ is an observable of diagrams of weight $|\mu|+\ell(\mu)$, and that
$$\varSigma_{\mu}=R_{\mu_{1}+1}\,R_{\mu_{2}+1}\,\cdots\,R_{\mu_{r}+1}+(\text{observable of weight strictly smaller than } |\mu|+\ell(\mu)).$$
As a consequence, the central characters $\varSigma_{\mu}$ form a linear basis of $\obs$ when $\mu$ runs over partitions, and the cyclic central characters $\varSigma_{k\geq 1}$ form an algebraic basis. In particular, this allows to consider central characters of continuous diagrams: it suffices to expand $\varSigma_{\mu}(\omega)$ in the basis of interlacing moments, and to use the general definition of these latter observables for continuous diagrams. So, to summarize:
$$\obs=\C[p_{2},p_{3},\ldots,p_{n},\ldots]=\C[R_{2},R_{3},\ldots,R_{n},\ldots]=\C[\varSigma_{1},\varSigma_{2},\ldots,\varSigma_{n},\ldots]$$\bigskip

Two other facts are worth being mentioned. First, if $f$ is an \emph{homogeneous} observable of weight $k$, then for any continuous diagram $\omega$ and any real parameter $t$, 
$$f(\omega^{t})=t^{-k/2}\,f(\omega).$$
On the other hand, the central characters $\varSigma_{\mu}$ can be interpreted as elements of the \textbf{algebra of Ivanov and Kerov} that is a subalgebra of the \textbf{algebra of partial permutations}, \emph{cf.} \cite{IK99}. Hence, as a linear combination of partial permutations,
$$\varSigma_{\mu}=\sum\,\, (a_{1,1},\ldots, a_{1,\mu_{1}})\,(a_{2,1},\ldots, a_{2,\mu_{2}})\,\cdots\,(a_{r,1},\cdots,a_{r,\mu_{r}})
$$
where the sum is taken on injective functions $a : \{(i,j) \,\,|\,\, 1 \leq i \leq \ell(\mu),\,\,1\leq j \leq \mu_{i}\} \to \N^{*}$, and a cycle $(a_{i,1},\ldots,a_{i,\mu_{i}})$ is attached to the support $\{a_{i,1},\ldots,a_{i,\mu_{i}}\}$. This interpretation provides a combinatorial way to deal with products of central characters (see in particular \cite[\S3.3]{FM10}), and moreover, it permits to define another filtration of algebra on $\obs$ given by \textbf{Kerov's degree}:
$$\deg_{K}(\varSigma_{\mu})=\mu+m_{1}(\mu)$$
This is indeed a filtration of algebra on $\obs$, see \cite[Proposition 10.1]{IK99} and \cite[Proposition 4.7]{IO02}; but in constrast to the weight filtration, a product of two central characters $\varSigma_{\lambda}\,\varSigma_{\mu}$ has not the same top homogeneous component for  Kerov's degree as $\varSigma_{\lambda \sqcup \mu}$. Actually, it is the case when $\lambda$ and $\mu$ have no part in common, see \cite[Proposition 4.13]{IO02}:
$$\varSigma_{\lambda}\,\varSigma_{\mu}=\varSigma_{\lambda\sqcup \mu} + (\text{observable of smaller Kerov degree})\quad\text{if }m_{i}(\lambda)\,m_{i}(\mu)=0 \text{ forall }i.$$
Later, we shall also describe the component of higher Kerov degree of a power $(\varSigma_{k})^{m}$, which is in some sense the opposite case of the previous situation. Another important result for the study of fluctuations of the Gelfand measures is the description of the higher Kerov degree component of the $k$-th interlacing moment:
$$p_{k}=\sum_{j=0}^{\lfloor \frac{k-3}{2}\rfloor}\frac{k^{\downarrow j+1}}{j!}\,\varSigma_{k-1-2j}\,(\varSigma_{1})^{j} +\begin{cases}\binom{k}{k/2}\,(\varSigma_{1})^{k/2}&\text{if }k \text{ is even},\\
0&\text{otherwise},\end{cases}$$
plus some observable of Kerov degree smaller than $k-2$. We refer to \cite[Proposition 7.3]{IO02} for a proof of this decomposition; it will be useful later for linking the asymptotics of observables of diagrams to the asymptotics of the actual shapes of the Young diagrams.
\bigskip
\bigskip

That said, the \textbf{method of noncommutative moments} in the setting of random partitions consists in studying the expectations $\esper[f(\lambda)]$ with $f$ in one of the three bases of $\obs$ presented above. In many situations, this is sufficient to understand the first and second order asymptotics, see \cite{IO02}, \cite{FM10}, \cite{Mel10b}. Moreover, given a model of random partitions of size $n$ coming from a linear representation of $\sym_{n}$, it is not difficult to compute the expectations of the central characters $\varSigma_{\mu}$ viewed as random variables of the random partitions $\lambda \in \Part_{n}$, because they are related to the trace of the representation. Hence, the method of noncommutative moments is extremely versatile, and we will see that it is indeed a powerful tool in the case of Gelfand measures. 
\bigskip
\bigskip

\section{Enumeration of square roots in the symmetric group and asymptotic distribution of the central characters}
In this section, we will prove that if $\lambda$ is picked randomly according to the Gelfand measure $\Gel_{n}$, then the rescaled central characters $\frac{\varSigma_{k}(\lambda)}{n^{k/2}}$ are asymptotically independent gaussian variables. There is a similar result in the case of Plancherel measures (\emph{cf.} \cite[Theorem 6.1]{IO02}), but our computations are quite different, and they involve the following problem\footnote{The formulas given in \cite{APR07} for these numbers seem to be false, unless we did not understand their notations of multinomial coefficients.}: given a permutation $\sigma \in \sym_{n}$, how many permutations $\tau$ are ``square roots of $\sigma$'', that is to say that $\tau^{2}=\sigma$?
\bigskip\bigskip

\subsection{Computation of the expectations of the central characters}
Let $\mu=1^{m_{1}}\,2^{m_{2}}\,\cdots \,s^{m_{s}}$ be a partition with $m_{1}$ parts of size $1$, $m_{2}$ parts of size $2$, \emph{etc.} If $n$ is an integer greater than $|\mu|$, then:
$$\Gel_{n}[\varSigma_{\mu}]=n^{\downarrow |\mu|} \,\Gel_{n}[\chi^{\lambda}(\mu\sqcup 1^{n-|\mu|})]=n^{\downarrow |\mu|} \,\Gel_{n}[\chi^{\lambda}(1^{n-|\mu|+m_{1}}\,2^{m_{2}}\,\cdots \,s^{m_{s}})]$$
and because of Proposition \ref{tracegelfand}, this expectation is related to the number of square roots in $\sym_{n}$ of a permutation of cycle type $1^{n-|\mu|+m_{1}}\,2^{m_{2}}\,\cdots \,s^{m_{s}}$:
$$\Gel_{n}[\varSigma_{\mu}]=n^{\downarrow |\mu|} \,\frac{\card\{\tau \in \sym_{n}\,\,|\,\,\tau^{2}=\sigma_{\mu\sqcup 1^{n-|\mu|}}\}}{I_{n}}$$
Let us determine precisely the cardinality of the set of square roots. The square of any odd cycle is a cycle of same length, whereas the square of any even cycle of length $2l$ is a product of two disjoint $l$-cycles. Consequently, if $\tau$ is a permutation of cycle type $1^{k_{1}}\,2^{k_{2}}\,\cdots \,s^{k_{s}}$, then the cycle type of $\tau^{2}$ is:
$$\left(1^{k_{1}}\,3^{k_{3}}\,5^{k_{5}}\,\cdots\right)\sqcup \left(1^{2k_{2}}\,2^{2k_{4}}\,3^{2k_{6}}\,\cdots\right)$$
So, a permutation $\sigma$ of cycle type $\mu$ is a square if and only if $m_{2i}(\mu)$ is even for any even integer $2i$. In that case, to choose a square root of $\sigma$, one has to:
\begin{enumerate}
\item For any odd integer $i$, choose the integers $k_{i}$ and $k_{2i}$ such that $m_{i}(\mu)=k_{i}+2k_{2i}$. Then, choose a partial matching of $2k_{2i}$ of the $i$-cycles, and for each pair $(c_{1},c_{2})$ in this matching, choose a $2i$-cycle whose square is $c_{1}c_{2}$; there are $i$ such cycles.
\item For any even integer $2i$, just choose a complete matching of the $m_{2i}(\mu)$ cycles of length $2i$, and for each pair $(c_{1},c_{2})$ in this matching, choose a $4i$-cycle whose square is $c_{1}c_{2}$; there are $2i$ such cycles.
\end{enumerate}
So, the correct formula (in contrast to \cite[Corollary 3.2]{APR07}) is given by:
\begin{proposition}[Number of square roots of a permutation]
If $\sigma$ has cycle type $\mu=1^{m_{1}}\,2^{m_{2}}\,\cdots\,s^{m_{s}}$, then $\card\{\tau \in \sym_{n}\,\,|\,\,\tau^{2}=\sigma\}=\prod_{i=1}^{s} f(i,m_{i})$, where:
$$f(i,m)=\begin{cases}0&\text{if }i\text{ is even and }m\text{ is odd},\\
\frac{m!}{m/2!}\,\left(\frac{i}{2}\right)^{m/2}&\text{if }i\text{ and }m\text{ are even},\\
\sum_{k=0}^{\lfloor \frac{m}{2}\rfloor}\frac{m!}{m-2k!\,k!}\, \left(\frac{i}{2}\right)^{k}&\text{if }i\text{ is odd}.\end{cases}$$ 
\end{proposition}\bigskip

\begin{example}
Suppose that $\sigma$ has cycle type $1^{3}\,2^{2}\,3$ in $\sym_{10}$; up to conjugacy, we can suppose that $\sigma=(1)(2)(3)(4,5)(6,7)(8,9,10)$. If $\tau^{2}=\sigma$, then the three fixed points $1,2,3$ correspond either to a component $(1)(2)(3)$ in $\tau$, or to one of the involutions $(1,2)(3)$, $(1,3)(2)$ or $(1)(2,3)$; whence $f(1,3)=4$ possibilities. The two $2$-cycles come from one of the $4$-cycles $(4,6,5,7)$ and $(4,7,5,6)$; whence $f(2,2)=2$ possibilities. Finally, the $3$-cycle $(8,9,10)$ comes necessarily from the $3$-cycle $(8,10,9)$ in $\tau$; whence $f(3,1)=1$ possibilities. So, the number of square roots of $\sigma$ in $\sym_{10}$ is
$$f(1,3)\,f(2,2)\,f(3,1)=8.$$
\end{example}
\bigskip

\noindent As a consequence of this enumeration, one sees that $$\Gel_{n}[\varSigma_{\mu}] = n^{\downarrow |\mu|}\,\frac{f(1,n-|\mu|+m_{1})}{f(1,n)}\,\left(\prod_{i=2}^{s}f(i,m_{i})\right)=n^{\downarrow |\mu|}\,\frac{I_{n-|\mu|+m_{1}}}{I_{n}}\,\left(\prod_{i=2}^{s}f(i,m_{i})\right),$$
so the asymptotics of the expectations $\Gel_{n}[\varSigma_{\mu}]$ can be deduced from those of the numbers of involutions $I_{n}$. The exponential generating function of these numbers is
$$I(z)=\sum_{n=0}^{\infty}\frac{I_{n}}{n!}\,z^{n}=\sum_{n=0}^{\infty}\sum_{k=0}^{\lfloor \frac{n}{2}\rfloor}\frac{1}{n-2k!\,k!}\,z^{n-2k}\,\left(\frac{z^{2}}{2}\right)^{k}=\exp\left(z+\frac{z^{2}}{2}\right),$$
and by saddle-point analysis (see \cite[Chapter VIII, p. 559]{FS09}), one obtains $I_{n}\simeq \left(\frac{n}{\E}\right)^{\frac{n}{2}}\,\frac{\E^{\sqrt{n}-1/4}}{\sqrt{2}}$ (this is a very classical estimate similar to Stirling formula).
\begin{theorem}[Asymptotic expectations of the central characters]\label{a}
When $n$ goes to infinity,
$$\Gel_{n}[\varSigma_{\mu}]\simeq \left(\prod_{i=2}^{s}f(i,m_{i})\right)\,n^{\frac{|\mu|+m_{1}(\mu)}{2}}.$$
\end{theorem}
\begin{proof}
Asymptotically, the term $n^{\downarrow |\mu|}$ can be replaced by $n^{|\mu|}$, and
$$\frac{I_{n-k}}{I_{n}}\simeq \left(\frac{n-k}{n}\right)^{\frac{n-k}{2}}\,\left(\frac{\E}{n}\right)^{\frac{k}{2}}\,\E^{\sqrt{n-k}-\sqrt{n}}\simeq \frac{1}{n^{\frac{k}{2}}}$$
with $k=|\mu|-m_{1}(\mu)$. This result implies in particular the following estimate: for any observable $f$, $\Gel_{n}[f]=O(n^{\deg_{K}(f)/2})$. Indeed, this is true for the central characters that form a linear basis of $\obs$. Since Kerov's degree is always smaller than the weight of observables, one has also $\Gel_{n}[f]=O(n^{\mathrm{wt}(f)/2})$ for any observable.
\end{proof}\bigskip\bigskip

\subsection{Asymptotic distribution of the scaled cyclic central characters} An important result related to Kerov's central limit theorem is the following: under the Plancherel measures $\proba_{n}$, the scaled central characters $\frac{\varSigma_{k}}{n^{k/2}}$ with $k \geq 2$ converge towards independent gaussian variables $\xi_{k}\sim \mathcal{N}(0,k)$ of variances $k$. The goal of this section is to prove the analogue of this result in the setting of Gelfand measures:
\begin{theorem}[Asymptotic distribution of the cyclic central characters]\label{nolove}
Under the Gelfand measures $\Gel_{n}$, the scaled central characters $\frac{\varSigma_{k}}{n^{k/2}}$ with $k \geq 2$ converge towards independent gaussian variables $\xi_{k}\sim \mathcal{N}(e_{k},2k)$ of variances $2k$ and expectations
$$e_{k}=\begin{cases} 1 &\text{if }k \text{ is odd},\\
 0 &\text{if }k \text{ is even}.
 \end{cases}$$
\end{theorem}\bigskip

\begin{lemma}[Component of highest Kerov degree of a power $(\varSigma_{k})^{m}$]\label{m}
For any integers $k\geq 2$ and $m \geq 1$, the component of highest Kerov degree of $(\varSigma_{k})^{m}$ is:
$$\sum_{p=0}^{\lfloor \frac{m}{2}\rfloor} \frac{m!}{m-2p!\,p!}\,\left(\frac{k}{2}\right)^{p}\,\varSigma_{1^{kp}\,k^{m-2p}}.$$
\end{lemma}
\begin{proof}
This lemma is essentially equivalent to the content of \cite[Propositions 4.11, 4.12, 6.2 and 6.3]{IO02}, and it can be proved in a pure combinatorial way if one considers the central characters as elements of the algebra of partial permutations. Kerov's degree can be lifted up to this algebra by setting:
$$\deg_{K}(\sigma,S)=\card(\mathrm{Fix}(\sigma)\cap S)+\card\,S$$
for any partial permutation $(\sigma,S)$. It yields a filtration of algebra, because
$$\deg_{K}((\sigma,S)(\tau,T))=\deg_{K}(\sigma\tau,S\cup T)=\big\{\card(\mathrm{Fix}(\sigma\tau)\cap (S\cup T))-\card\,S\cap T\big\}+\card\,S+\card\,T,$$
and the term between brackets is smaller than $\card (\mathrm{Fix}(\sigma) \cap S)+\card(\mathrm{Fix}(\tau)\cap T)$. Indeed, if $x \in S\cup T$ satisfy $\sigma\tau(x)=x$, then:
\begin{enumerate}
\item The element $x$ can be in $S\setminus T$, but in that case $\tau(x)=x$, so $\sigma(x)=x$ and $x \in \mathrm{Fix}(\sigma)\cap S$.
\item Similarly, the element $x$ can be in $T \setminus S$, and in that case it is also in $\mathrm{Fix}(\tau)\cap T$.
\item Finally, if the two previous assumptions are false, then $x$ is in $S\cap T$.
\end{enumerate}
One concludes that $\card(\mathrm{Fix}(\sigma\tau)\cap (S\cup T) \leq \card (\mathrm{Fix}(\sigma)\cap S) + \card(\mathrm{Fix}(\tau)\cap T)+\card(S\cap T)$, so
$$\deg_{K}((\sigma,S)(\tau,T))\leq \deg_{K}(\sigma,S)+\deg_{K}(\tau,T).$$
As an element of the algebra of partial permutations, 
$$(\varSigma_{k})^{m}=\sum (c_{1},S_{1})(c_{2},S_{2})\cdots(c_{m},S_{m})$$
where the sum is taken over $m$-tuples of $k$-arrangements $(a_{i1},\ldots,a_{ik})$ giving a cycle $c_{i}=(a_{i1},\ldots,a_{ik})$ and a support $S_{i}=\{a_{i1},\ldots,a_{ik}\}$. Since $\deg_{K}(\varSigma_{k})=k$, $\deg_{K}((\varSigma_{k})^{m}) \leq km$, and this is an identity because this is the Kerov degree of a product $(c_{1},S_{1})\cdots(c_{m},S_{m})$ of \emph{disjoint} $k$-cycles. Now, let us collect the products that have indeed this maximal degree. If $\deg_{K}((c_{1},S_{1})\cdots(c_{m},S_{m}))=km$, then since $\deg_{K}$ is a filtration of algebra, one has necessarily:
\begin{align*}
\deg_{K}((c_{1},S_{1})(c_{2},S_{2}))&=2k\\
\deg_{K}((c_{1},S_{1})(c_{2},S_{2})(c_{3},S_{3}))&=3k\\
\vdots\qquad\qquad\qquad\qquad & \vdots\\
\deg_{K}((c_{1},S_{1})(c_{2},S_{2})\cdots (c_{i},S_{i}))&=ik\\
\vdots\qquad\qquad\qquad\qquad & \vdots
\end{align*}
Let us consider the two first terms $(c_{1},S_{1})$ and $(c_{2},S_{2})$. We have shown that:
\begin{align*}2k&=\deg_{K}((c_{1},S_{1})(c_{2},S_{2}))=\big\{\card((\mathrm{Fix}(c_{1}c_{2})\cap (S_{1}\cup S_{2})) -\card\, S_{1}\cap S_{2}\big\}+\card\, S_{1}+\card\, S_{2}\\
&=2k+\big\{\card((\mathrm{Fix}(c_{1}c_{2})\cap (S_{1}\cup S_{2})) -\card\, S_{1}\cap S_{2}\big\}\end{align*}
so each point of the intersection $S_{1}\cap S_{2}$ is a fixed point of $c_{1}c_{2}$. In the following, $u$ is the exponent of $c_{1}c_{2}$, so $(c_{1}c_{2})^{u}=\id$. Let us suppose that $S_{1}$ and $S_{2}$ have a non-empty intersection, but are not equal sets. We denote $c_{1}=(a_{1},\ldots,a_{k})$ and $c_{2}=(b_{1},\ldots,b_{k})$; by hypothesis, there is some $b_{j}$ that is in $S_{2}$, but not in $S_{1}$. We iterate $c_{1}c_{2}$ on this element. If $b_{j+1}$ is not in $S_{1}$, then $c_{1}c_{2}(b_{j})=c_{1}(b_{j+1})=b_{j+1}$, and we can go on with $b_{j+1}$, $b_{j+2}$, \emph{etc.} Hence, we construct a chain $b_{j},b_{j+1}, \ldots,b_{J-1}$ such that all elements are in $S_{2}\setminus S_{1}$, and therefore satisfy $c_{1}c_{2}(b_{p})=b_{p+1}$. But some point $b_{J}$ should be in $S_{1} \cap S_{2}$, because the two supports have a non-empty intersection. Then, $a_{i}=c_{1}(b_{J})$ is a certain $(c_{1}c_{2})^{s}(b_{j})$, and is in $S_{1}$. It cannot be in $S_{1}\cap S_{2}$, because in this case all the $(c_{1}c_{2})^{t}(a_{i})$ are equal to $a_{i}$, and by taking $t=u-s$ we then recover $b_{j}=a_{i}$, that is not in $S_{1}$. So, $a_{i}$ is in $S_{1}\setminus S_{2}$. We continue to iterate $c_{1}c_{2}$ on $a_{i}$; since $a_{i}$ is not in $S_{2}$, $c_{1}c_{2}(a_{i})=a_{i+1}$, and we go on until some $a_{I}$ that is in $S_{1}$ and $S_{2}$ (see figure \ref{absurd}). 

\figcap{\includegraphics{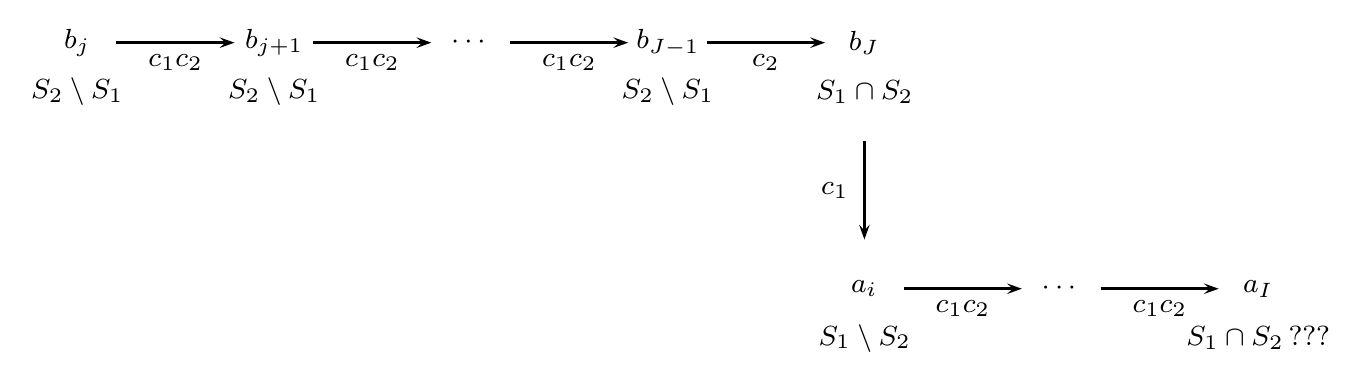}}{If $S_{1} \neq S_{2}$ and $S_{1}\cap S_{2}\neq \emptyset$, then the identity $\deg_{K}((c_{1},S_{1})(c_{2},S_{2}))=2k$ yields a contradiction.\label{absurd}}
\noindent This $a_{I}$ is a certain $(c_{1}c_{2})^{s'}(b_{j})$, and is a fixed point of $c_{1}c_{2}$. Then, by taking $a_{I}=(c_{1}c_{2})^{t'}(a_{I})$ with $t'=u-s'$, we recover $b_{j}$, and that is a contradiction because $b_{j}$ is not in $S_{1}\cap S_{2}$.
\bigskip

We conclude that if $(c_{1},S_{1})(c_{2},S_{2})$ has Kerov degree $2k$, then either $S_{1}\cap S_{2}=\emptyset$ (and this corresponds to an element of cycle type $(k,k)$), or $S_{1}=S_{2}$. Then, $\card\, S_{1}\cup S_{2}=k$, so the number of fixed points of $c_{1}c_{2}$ has to be $k$, which means that $c_{1}=c_{2}^{-1}$. We conclude that $c_{2}$ has to be either disjoint from $c_{1}$, or equal to $c_{1}^{-1}$. For the same reasons, $c_{3}$ has to be either disjoint from $c_{1}$ and $c_{2}$, or equal to the inverse of one of these cycles, assuming in that case that $c_{1}\neq c_{2}^{-1}$. By induction on $m$, we conclude that a cycle $c_{i}$ involved in a product of Kerov degree $km$ is either disjoint of all the other cycles, or paired with one (and only one) cycle that is its inverse. So, the cycle type of a product of maximal Kerov degree is always some $1^{kp}\,k^{m-2p}$,
where $p$ corresponds to the number of pairs of inverse cycles. This integer $p$ being fixed, we have a factor $$\frac{m!}{m-2p!\,p!\,2^{p}}$$ before the term $\varSigma_{1^{kp}\,k^{m-2p}}$ corresponding to the number of partial matchings with $p$ pairs; and there will also be a factor $k^{p}$ in the enumeration, because given an arrangement $(a_{i1},\ldots,a_{ik})$, there is $k$ cyclically different arrangements $(a_{j1},\ldots,a_{jk})$ that correspond to the inverse cycle. So, we conclude that the term of higher Kerov degree in $(\varSigma_{k})^{m}$ is indeed
$$\sum_{p=0}^{\lfloor \frac{m}{2}\rfloor} \frac{m!}{m-2p!\,p!}\,\left(\frac{k}{2}\right)^{p}\,\varSigma_{1^{kp}\,k^{m-2p}}.\vspace{-5mm}$$
\end{proof}
\begin{example}
$(\varSigma_{4})^{4}=\varSigma_{4^{4}}+24\,\varSigma_{1^{4}\, 4^{2}}+48\,\varSigma_{1^{8}}+(\text{terms of Kerov degree strictly smaller than }16)$.
\end{example}
\bigskip

\begin{proof}[Proof of Theorem \ref{nolove}, first part] We shall first prove the \emph{non-joint} convergence in law of the scaled central characters; the asymptotic independence will be proved later. Suppose that $k$ is an even integer and $m$ is odd. Then, the terms $\varSigma_{1^{kp}\,k^{m-2p}}$ involved in $(\varSigma_{k})^{m}$ have null expectations under the Gelfand measures, because $m-2p$ is always odd, and the permutations with these cycle types have no square roots. Consequently, $\Gel_{n}[(\varSigma_{k})^{m}]$ is in that case equal to the expectation of a term of Kerov degree less than $km-1$, so:
$$\Gel_{n}\left[\left(\frac{\varSigma_{k}}{n^{k/2}}\right)^{m}\right]=\frac{O(n^{km-1/2})}{n^{km/2}}=O(n^{-1/2})\to 0$$
Now, suppose that $k$ and $m$ are even. Then, the expectations of the terms of Kerov degree less than $km-1$ are negligible in the asymptotics of $\Gel_{n}\left[\left(\frac{\varSigma_{k}}{n^{k/2}}\right)^{m}\right]$, so:
\begin{align*}
\Gel_{n}\left[\left(\frac{\varSigma_{k}}{n^{k/2}}\right)^{m}\right]&\simeq \sum_{p=0}^{\lfloor \frac{m}{2}\rfloor} \frac{m!}{m-2p!\,p!}\,\left(\frac{k}{2}\right)^{p}\,\Gel_{n}\left[\frac{\varSigma_{1^{kp}\,k^{m-2p}} }{n^{km/2}}\right]=\sum_{p=0}^{\lfloor \frac{m}{2}\rfloor} \frac{m!}{m-2p!\,p!}\,\left(\frac{k}{2}\right)^{p}\,f(k,m-2p)\\
&\simeq\left(\frac{k}{2}\right)^{m/2}\,\sum_{p=0}^{\lfloor \frac{m}{2}\rfloor} \frac{m!}{m/2-p!\,p!}=k^{m/2}\,\frac{m!}{m/2!}=(2k)^{m/2}\,(m-1!!)
\end{align*}
But it is well-known that a standard gaussian variable $X$ of variance $1$ is characterized by the vanishing of all odd moments, and by the identity $\esper[X^{m}]=m-1!!$ for any even number $m$. Since gaussian variables are characterized by their moments, we conclude that if $k$ is even, one has the convergence in law
$$\frac{\varSigma_{k}}{n^{k/2}} \to \mathcal{N}(0,2k) $$
under the Gelfand measures $\Gel_{n}$. The case when $k$ is odd is slightly more complicated, but essentially similar. For the same reasons as before, one can neglict the terms of Kerov degree less than $km-1$ in the asymptotics of the $m$-th moment, so:
\begin{align*}
\Gel_{n}\left[\left(\frac{\varSigma_{k}}{n^{k/2}}\right)^{m}\right]&\simeq \sum_{p=0}^{\lfloor \frac{m}{2}\rfloor} \frac{m!}{m-2p!\,p!}\,\left(\frac{k}{2}\right)^{p}\,\Gel_{n}\left[\frac{\varSigma_{1^{kp}\,k^{m-2p}} }{n^{km/2}}\right]=\sum_{p=0}^{\lfloor \frac{m}{2}\rfloor} \frac{m!}{m-2p!\,p!}\,\left(\frac{k}{2}\right)^{p}\,f(k,m-2p)\\
&\simeq \sum_{p=0}^{\lfloor \frac{m}{2}\rfloor}\sum_{q=0}^{\lfloor \frac{m-2p}{2}\rfloor} \frac{m!}{p!\,q!\,m-2(p+q)!}\,\left(\frac{k}{2}\right)^{p+q}=\sum_{r=0}^{\lfloor \frac{m}{2}\rfloor} \frac{m!}{m-2r!} \,\left(\frac{k}{2}\right)^{r}\,\left(\sum_{p+q=r}\frac{1}{p!\,q!}\right)\\
&\simeq \sum_{r=0}^{\lfloor\frac{m}{2}\rfloor} \frac{m!}{m-2r!\,r!}\,k^{r}
\end{align*}
Let us consider the generating function corresponding to these asymptotics moments. It is given by
$$\sum_{m=0}^{\infty} \sum_{r=0}^{\lfloor\frac{m}{2}\rfloor} \frac{1}{m-2r!\,r!}\,k^{r}\,z^{m}=\sum_{m=0}^{\infty} \sum_{r=0}^{\lfloor\frac{m}{2}\rfloor} \frac{1}{m-2r!\,r!}\,(kz^{2})^{r}\,z^{m-2r}=\exp\left(z+kz^{2}\right),$$
and this is the generating function $\esper[\E^{zX}]$ of a gaussian random variable of law $\mathcal{N}(1,2k)$. Hence, one has the convergence in law
$$\frac{\varSigma_{k}}{n^{k/2}} \to \mathcal{N}(1,2k) $$
under the Gelfand measures. 
\end{proof}\bigskip
\bigskip

\subsection{Asymptotic independence of the scaled cyclic central characters}\label{asymptoticindependence}

Another way to prove the asymptotic gaussian behaviour of the $\varSigma_{k}$'s is to look at the joint cumulants of these observables, and to prove strong asymptotic estimates for the cumulants of order greater than $3$ (see \cite{Sni06}). Recall that if $X_{1},\ldots,X_{r}$ are random variables defined on a same probability space, then their joint cumulant is defined by:
\begin{align*}k(X_{1},\ldots,X_{r})&=\left.\frac{\partial^r}{\partial t_{1}\cdots \partial t_{r}}\right|_{t=0} \exp\left(t_{1}X_{1}+\cdots+t_{r}X_{r}\right)\\
&=\sum_{\pi \in \mathfrak{Q}(\lle 1,r\rre)} \mu(\pi)\,\prod_{\pi_{j} \in \pi} \esper\left[\prod_{i \in \pi_{j}} X_{i}\right]\end{align*}
where $\mathfrak{Q}(\lle 1,r\rre)$ is the set of set partitions of the interval $\lle 1,r\rre$, and $\mu$ is the M\"obius function of this lattice with respect to his maximum, that is to say that
$$\mu(\pi)=(-1)^{\ell(\pi)-1}\,(\ell(\pi)-1)!$$
where $\ell(\pi)$ is the number of parts of $\pi$. Thus, the cumulant $k(X_{1})$ of a single random variable is its expectation $\esper[X_{1}]$, and the cumulant $k(X_{1},X_{2})$ of two random variables is $\esper[X_{1}X_{2}]-\esper[X_{1}]\,\esper[X_{2}]$, that is to say their covariance. A family of random variables is a gaussian vector if and only if all the joint cumulants of order $r\geq 3$ of these variables are equal to zero. Hence, in order to prove the second part of Theorem \ref{nolove}, it is sufficient to establish the following:
$$\forall r\geq 3,\,\,\,\forall l_{1},\ldots,l_{r} \geq 2,\,\,\,k\left(\frac{\varSigma_{l_{1}}}{n^{l_{1}/2}},\frac{\varSigma_{l_{2}}}{n^{l_{2}/2}},\ldots,\frac{\varSigma_{l_{r}}}{n^{l_{r}/2}}\right)=o(1)\quad;\quad\forall l\neq m \geq 2,\,\,\,k\left(\frac{\varSigma_{l}}{n^{l/2}},\frac{\varSigma_{m}}{n^{m/2}}\right)=o(1)$$
Indeed, this will implie that the limiting gaussian variables $\xi_{k}$ form a gaussian vector, and also that they are independent.\bigskip\bigskip

Since the $\varSigma_{k}$'s with $k \geq 2$ have Kerov degree $k$, the expression above is already a $O(1)$, because it can be expanded as a finite combination of products of expectations of homogeneous observables $f$ rescaled by $n^{\deg_{K}(f)/2}$. So, one has to win just one order of magnitude. Let us introduce as in \cite{Sni06} and \cite{FM10} the \textbf{disjoint cumulants} and the \textbf{identity cumulants} of observables:
\begin{align*}k^{\bullet}(X_{1},\ldots,X_{r})&=\sum_{\pi \in \mathfrak{Q}(\lle 1,r\rre)} \mu(\pi)\,\prod_{\pi_{j} \in \pi} \esper\left[\prod_{i \in \pi_{j}}^{\bullet} X_{i}\right]\\
 k^{\id}(X_{1},\ldots,X_{r})&=\sum_{\pi \in \mathfrak{Q}(\lle 1,r\rre)} \mu(\pi)\,\prod_{\pi_{j} \in \pi}^{\bullet} \left[\prod_{i \in \pi_{j}} X_{i}\right]\end{align*}
where the disjoint product of observables $\bullet$ of diagrams is defined on the basis of central characters by $\varSigma_{\mu}\bullet \varSigma_{\nu}=\varSigma_{\mu \sqcup \nu}$. Since Kerov's degree is compatible with the usual product of observables and also with the disjoint product of observables, for any family of observables, 
$$k(X_{1},\ldots,X_{r}) =O\left(n^{\frac{\deg_{K}(X_{1})+\cdots+\deg_{K}(X_{r})}{2}}\right)\qquad;\qquad k^{\bullet}(X_{1},\ldots,X_{r}) =O\left(n^{\frac{\deg_{K}(X_{1})+\cdots+\deg_{K}(X_{r})}{2}}\right).$$ 
On the other hand, because of the commutativity of the following diagram of noncommutative probability spaces
$$\begin{diagram}
   \node{(\obs,\cdot)} \arrow{se,b}{\esper=\Gel_{n}}\arrow[2]{e,t}{\esper^\id=\id} \node[2]{(\obs,\bullet)}\arrow{sw,b}{\esper=\Gel_{n}}\\
\node[2]{\C}
  \end{diagram}
$$
the three kind of cumulants are related by the following formula:
$$k(X_{1},\ldots,X_{r})=\sum_{\pi \in \mathfrak{Q}(\lle 1,r\rre)} k^{\bullet}\left(k^{\id}(X_{i \in \pi_{1}}), \ldots, k^{\id}(X_{i \in \pi_{s}}) \right)$$
Therefore, to gain one order of magnitude in the estimate of $k(\varSigma_{l_{1}},\ldots,\varSigma_{l_{r}})$, one can use this expansion in disjoint cumulants and:
\begin{enumerate}
\item either bound the sum of the Kerov degrees of the identity cumulants by $l_{1}+\cdots+l_{r}-1$;
\item or, in case the sum of the Kerov degrees is maximal and equal to $l_{1}+\cdots+l_{r}$, bound the disjoint cumulant of the identity cumulants.
\end{enumerate}\bigskip

\begin{lemma}[Kerov degree of a disjoint cumulant of observables]\label{e}
Given integers $l_{1},l_{2},\ldots,l_{t}$ greater than $2$, if these integers are not all equal, then
$$\deg_{K}\left(k^{\id}(\varSigma_{l_{1}},\ldots,\varSigma_{l_{t}})\right)\leq l_{1}+\cdots+l_{t}-1\,.$$
\end{lemma}
\begin{proof}
As before, we interpret the $\varSigma_{k}$'s as combinations of partial permutations, and we write formally:
$$\varSigma_{k}=\sum_{A \in \mathcal{A}(k)}c(A)$$
where $\mathcal{A}(k)$ is the set of $k$-arrangements of positive integers in $\N^{*}$, and $c(A)$ denotes the cyclic partial permutation associated to such an arrangement. If $A_{1},\ldots,A_{t}$ is a familly of arrangements of integers, we denote by $\pi(A_{1},\ldots,A_{t})$ the set partition of $\lle 1,t\rre$ associated to the equivalence relation that covers the relations $i \sim j \iff A_{i}\cap A_{j} \neq \emptyset$. Then, it has been shown in \cite{FM10} that identity cumulants of symbols $\varSigma_{l_{1}},\ldots,\varSigma_{l_{t}}$ can be given the following combinatorial interpretation:
$$k^{\id}(\varSigma_{l_{1}},\ldots,\varSigma_{l_{t}})= \sum_{\substack{\forall j,\,\,A_{j} \in \mathcal{A}(l_{j})\\ \pi(A_{1},\ldots,A_{t})=\{\lle 1,t \rre\}}} c(A_{1})\,c(A_{2})\,\cdots\,c(A_{t})$$
In other words, the arrangements appearing in the disjoint cumulant are bound to be ``evercrossing''. Now, since the $l_{i}$'s are not all equal, up to a permutation that does not modify the value of the joint cumulant, one can suppose that there is an index $u<t$ such that $l_{1}=l_{2}=\cdots=l_{u}$, and $l_{1}\neq l_{i}$ for all $i>u$. Then, let us consider a product of cycles $c(A_{1})\cdots c(A_{t})$ appearing in the sum above. If the partial product $c(A_{1})\cdots c(A_{u})$ has not maximal Kerov degree $u\,l_{1}$, then since Kerov's degree yields a filtration on the algebra of partial permutations, one has indeed
$$\deg_{K}(c(A_{1})\,\cdots\,c(A_{t}))\leq l_{1}+\cdots+l_{t}-1\,.$$
In the opposite case, because of the proof of Lemma \ref{m}, the cyclic type of the product of the $u$ first cycles is some $l_{1}^{u-2k}\,1^{l_{1}k}$. On the other hand, there is an index $v \in \lle u+1,t\rre$ such that the support of the cycle $c(A_{v})$ intersects the support of the products of the $u$ first cycles; we take the first such index. Notice that the length $l_{v}$ does not appear in the cycle type $l_{1}^{u-2k}\,1^{l_{1}k}$. But we have mentioned that if two partial permutations have cycle types without common part, then they have either disjoint supports, or a product with Kerov degree strictly less than the sum of the Kerov degrees. As the supports are supposed to have a non trivial intersection, we conclude that the partial product $c(A_{1})\,\cdots\,c(A_{v})$ has Kerov degree smaller than $l_{1}+\cdots+l_{v}-1$, and if one adds the remaining cycles, one obtains once again:
$$\deg_{K}(c(A_{1})\,\cdots\,c(A_{t}))\leq l_{1}+\cdots+l_{t}-1\,.$$
All the products of cycles involved in the sum satisfy then the inequality stated previously.
\end{proof}
\bigskip

\begin{lemma}[A M\"obius inversion formula]\label{l}
Let $F$ be any function on pairs of positive integers, and let $(1,\ldots,1,2,\ldots,2,\ldots,s,\ldots,s)$ be a sequence of $r$ integers with $r_{1}\geq 1$ integers $1$, $r_{2}\geq 1$ integers $2$, \emph{etc}. If $\pi \in \mathfrak{Q}(\lle 1,r\rre)$ is a set partition with parts $\pi_{1}\sqcup \pi_{2}\sqcup \cdots \sqcup \pi_{\ell(\pi)}$, we denote by $r_{ij}$ the number of integers $i$ that fall in $\pi_{j}$; in particular, $|\pi_{j}|=\sum_{i=1}^{s} r_{ij}$ and $r_{i}=\sum_{j=1}^{\ell(\pi)} r_{ij}$. Suppose that $s\geq 2$. Then, 
$$\sum_{\pi \in \mathfrak{Q}(\lle 1,r\rre)} (-1)^{\ell(\pi)-1}\,(\ell(\pi)-1)!\,\prod_{j=1}^{\ell(\pi)}\prod_{r_{ij}\geq 1} F(i,r_{ij})=0\,.$$
\end{lemma}
\begin{example}
Suppose that $r=4$. In that case, there are $15$ set partitions: $\{1,2,3,4\}$, $\{1,2,3\}\sqcup\{4\}$, $\{1,2,4\}\sqcup\{3\}$, $\{1,3,4\}\sqcup\{2\}$, $\{2,3,4\}\sqcup\{1\}$, $\{1,2\}\sqcup\{3,4\}$, $\{1,3\}\sqcup\{2,4\}$, $\{1,4\}\sqcup\{2,3\}$, $\{1,2\}\sqcup\{3\}\sqcup\{4\}$, $\{1,3\}\sqcup\{2\}\sqcup\{4\}$, $\{1,4\}\sqcup\{2\}\sqcup\{3\}$, $\{2,3\}\sqcup\{1\}\sqcup\{4\}$, $\{2,4\}\sqcup\{1\}\sqcup\{3\}$, $\{3,4\}\sqcup\{1\}\sqcup\{2\}$ et $\{1\}\sqcup\{2\}\sqcup\{3\}\sqcup\{4\}$. Suppose that $s=2$ and that the sequence is $(1,1,2,2)$. The corresponding sum is 
\begin{align*}
&F(1,2)\,F(2,2)-F(1,2)\,F(2,1)^{2}-F(1,2)\,F(2,1)^{2}-F(1,1)^{2}\,F(2,2)-F(1,1)^{2}\,F(2,2)\\
&-F(1,2)\,F(2,2)-F(1,1)^{2}\,F(2,1)^{2}-F(1,1)^{2}\,F(2,1)^{2}+2\,F(1,2)\,F(2,1)^{2}+2\,F(1,1)^{2}\,F(2,1)^{2}\\
&+2\,F(1,1)^{2}\,F(2,1)^{2}+2\,F(1,1)^{2}\,F(2,1)^{2}+2\,F(1,1)^{2}\,F(2,1)^{2}+2\,F(1,1)^{2}\,F(2,2)-6\,F(1,1)^{2}\,F(2,1)^{2}
\end{align*}
and is indeed equal to $0$.
\end{example}\bigskip

\begin{proof}[Proof of Theorem \ref{nolove}, second part]
We want to prove that $k(\varSigma_{l_{1}},\ldots,\varSigma_{l_{r}})=o(n^{l_{1}+\cdots+l_{r}/2})$ when $r\geq 3$. Let us denote $L=l_{1}+\cdots+l_{r}$. Because of the \emph{non-joint} convergence, the case when all the $l_{i}$'s are equal has already been proved; consequently, one can suppose that some of the $l_{i}$'s are distinct. We then use the expansion of the cumulant in terms of disjoint cumulants
$$k_{\pi}= k^{\bullet}\left(k^{\id}(\varSigma_{l_{i} \in \pi_{1}}), \ldots, k^{\id}(\varSigma_{l_{i} \in \pi_{s}}) \right)$$
labelled by set partitions. If some part $\pi_{j}$ contains distincts integers $l_{i}$'s, then because of lemma \ref{e}, the sum of the Kerov degrees of the observables involved in the disjoint cumulant $k_{\pi}$ is smaller than $L-1$, so $k_{\pi}=O(n^{(L-1)/2})$ is of smaller order of magnitude. Thus, we can focus on the remaining terms $k_{\pi}$, where $\pi$ is a set partition such that if $i_{1}$ and $i_{2}$ are in the same part of $\pi$, then $l_{i_{1}}=l_{i_{2}}$. In other words, we have to bound disjoint cumulants that look like
$$k^{\bullet}\left(k^{\id}(\varSigma_{m_{1}},\ldots,\varSigma_{m_{1}}),\ldots,k^{\id}(\varSigma_{m_{s}},\ldots,\varSigma_{m_{s}})\right).$$
If one of the identity cumulant $k^{\id}(\varSigma_{m_{j}},\ldots,\varSigma_{m_{j}})$ contains $n_{j}$ terms with $n_{j} \geq 3$, then the simple convergence of scaled central characters previously established ensures that  Kerov's degree of this identity cumulant is smaller than $m_{j}\,n_{j}-1$, so once again $k_{\pi}$ is a $O(n^{(L-1)/2})$. Thus, one can restrict again the set of summation and suppose that all the $n_{j}$'s are equal to $1$ or $2$. Of course, if $n_{j}=1$, then $k^{\id}(\varSigma_{m_{j}})=\varSigma_{m_{j}}$; on the other hand, because of Lemma \ref{m}, the top Kerov degree component of $k^{\id}(\varSigma_{m_{j}},\varSigma_{m_{j}})$ is $\varSigma_{1^{m_{j}}}$, and one can of course neglict the terms of smaller Kerov degree. In the end, up to a $O(n^{(L-1)/2})$, the joint cumulant we wanted to estimate is equal to a sum of disjoint cumulants
$$k^{\bullet}(\varSigma_{m_{1}},\ldots,\varSigma_{m_{u}},\varSigma_{1^{m_{u+1}}},\ldots,\varSigma_{1^{m_{u+v}}})$$
with $m_{1}+\cdots+m_{u}+2(m_{u+1}+\cdots+m_{u+v})=L$. Then, Lemma \ref{l} ensures that these disjoint cumulants are some $O(n^{(L-1)/2})$. Indeed, let us rewrite such a disjoint cumulant in the following way:
$$k^{\bullet}\left(\underbrace{\varSigma_{2},\ldots,\varSigma_{2}}_{r_{1}\text{ terms}},\ldots,\underbrace{\varSigma_{s+1},\ldots,\varSigma_{s+1}}_{r_{s}\text{ terms}},\underbrace{\varSigma_{1^{2}},\ldots,\varSigma_{1^{2}}}_{r_{s+1}\text{ terms}},\ldots,\underbrace{\varSigma_{1^{t+1}},\ldots,\varSigma_{1^{t+1}}}_{r_{s+t}\text{ terms}}\right)$$
with $2r_{1}+\cdots+(s+1)r_{s}+2(2r_{s+1}+\cdots+(t+1)r_{s+t})=L$. The hypothesis made at the beginning of the proof ensures that $s+t\geq 2$. Let us denote by $F$ the function defined by:
 $$F(i,r)=\begin{cases} f(i+1,r)\,n^{\frac{(i+1)r}{2}} &\text{if }i \in \lle 1,s\rre,\\
n^{(i+1-s) r}& \text{if }i \in \lle s+1,s+t \rre.
\end{cases}$$
Then, because of Theorem \ref{a} and the asymptotic expression of the expectations $\Gel_{n}[\varSigma_{\mu}]$, by using the M\"obius formula 
$$k(X_{1},\ldots,X_{r})=\sum_{\pi \in \mathfrak{Q}(\lle 1,r\rre)} (-1)^{\ell(\pi)-1}\,(\ell(\pi)-1)! \, \prod_{\pi_{j }\in \pi} \esper\left[\prod_{i \in \pi_{j}}X_{j}\right],$$ 
one sees that the disjoint cumulant $k_{\pi}$ has for term of degree $L/2$ in $n$:
$$\sum_{\pi \in \mathfrak{Q}(\lle 1, s+t\rre)}(-1)^{\ell(\pi)-1}\,(\ell(\pi)-1)! \, \prod_{j=1}^{\ell(\pi)}\, \prod_{r_{ij}\geq 1} F(i,r_{ij})$$
This is zero because of Lemma \ref{l}. Hence, all the $k_{\pi}$'s are $O(n^{(L-1)/2})$, and this ends the proof of the estimate. Then, one has the \emph{joint} convergence of the scaled central cyclic characters towards a gaussian vector, and the asymptotic independence will follow from the vanishing of the limiting covariances. But the proof above make use of the hypothesis $r\geq 3$ only in order to exclude the case when all the $l_{i}$'s are equal. So, if $l \neq m$, then the joint cumulant $k(\varSigma_{l},\varSigma_{m})$ is again of smaller order of magnitude $O(n^{\frac{l+m}{2}})$. So, the covariances of the scaled observables tend towards zero, and the asymptotic independence is finally established. 
\end{proof}
\bigskip

In \cite{IO02}, Ivanov and Olshanski used a totally different method in order to obtain the joint convergence, that relies on Hermite polynomials; at the opposite, our approach is purely combinatorial, and we hope that Lemmas \ref{e} and \ref{l} can be used in a general setting in order to precise the scope of application of Kerov's central limit theorem. Notice that since $\obs=\C[\varSigma_{1},\varSigma_{2},\ldots]$, Theorem \ref{nolove} describes in fact the asymptotic behaviour of any observable of diagrams. By switching to the other bases of $\obs$, we shall then understand the asymptotic behaviour of the shapes of the Young diagrams. 
\bigskip\bigskip

\section{Asymptotics of the shapes of the Young diagrams under the Gelfand measures}
To begin with, let us pick randomly a Young diagram under the Gelfand measure of size $n=500$. Notice that the easiest way to do this is to pick a random involution of size $n$, and then apply the RSK algorithm to this involution.
\figcap{\includegraphics{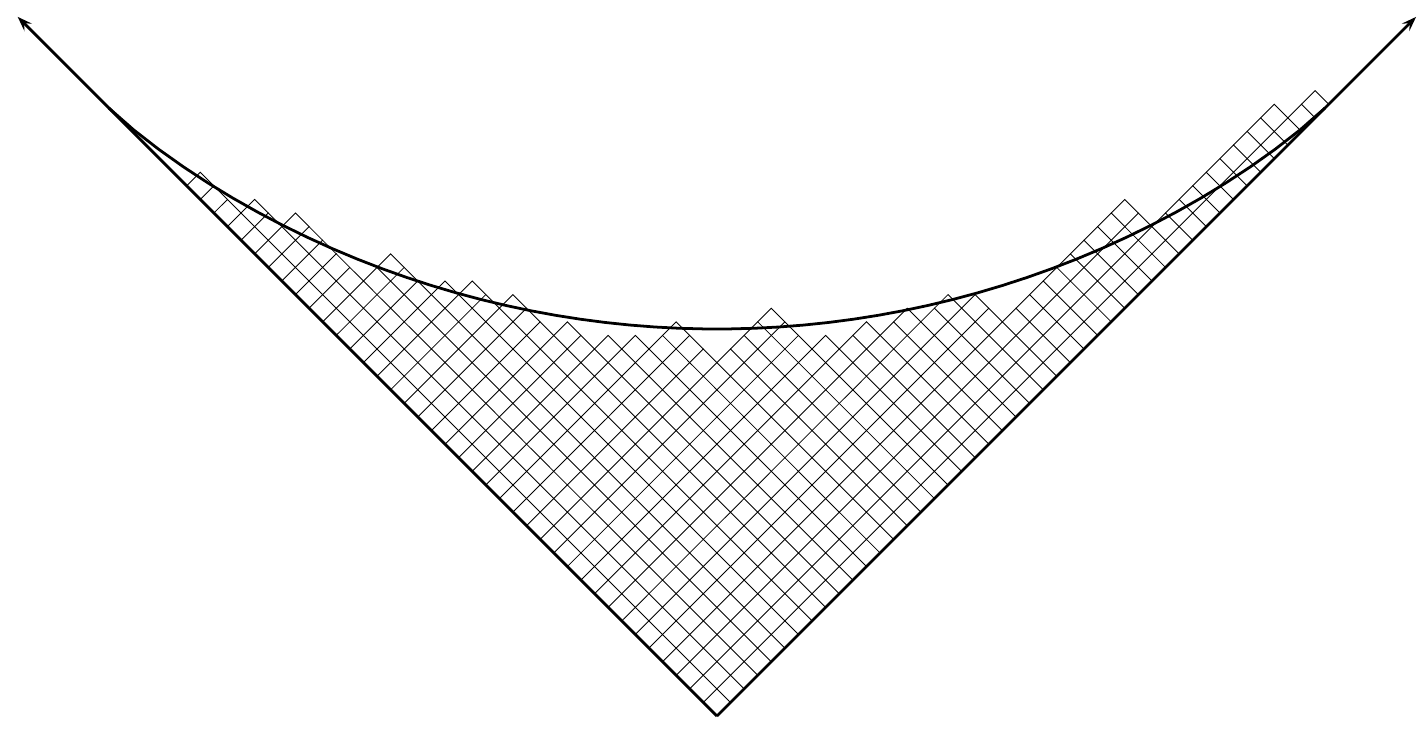}}{Young diagram of size $n=500$ picked randomly according to the Gelfand measure of the symmetric group.}
We see that the asymptotic behaviour seems to be exactly the same (at least at first order) as in the case of Plancherel measures. However, the fluctuations seem bigger\footnote{This may not be visible to the naked eye, and we shall see that the difference is only a factor $\sqrt{2}$. To highlight this phenomenon, we advise the reader to print the two figures and to color the areas between the diagrams and their limit shape. Then, it is clear that the area is bigger in the case of Gelfand measures; with the given figures, we can roughly estimate the areas to be of 20 small boxes in the case of the Plancherel measure, and of 40 small boxes in the case of the Gelfand measure.} (although we cannot exclude at this point the possibility that our sample exhibit non typical fluctuations). It is the purpose of this section to give rigourous proofs of these observations.\bigskip

\subsection{First order asymptotics and a variational argument}
The estimates on the expectations of the central characters implie the convergence in probability of the free cumulants:
\begin{theorem}[Asymptotics of the free cumulants of rescaled diagrams under the Gelfand measures]
Let $\lambda^{*}$ be the rescaled version of a random Young diagram picked randomly according to the probability measure $\Gel_{n}$. As $n$ goes to infinity,
$$R_{k}(\lambda^{*}) \to \begin{cases} 1&\text{if }k=2,\\
0&\text{otherwise}, \end{cases}$$
in probability.
\end{theorem}
\begin{proof}
Recall that for any observable $f$, $\Gel_{n}[f]$ is a $O(n^{\mathrm{wt}(f)/2})$. Since $R_{k+1}=\varSigma_{k}+h$ with $\mathrm{wt}(h)\leq k$, $R_{k+1}(\lambda^{*})=n^{-\frac{k+1}{2}}\,R_{k+1}(\lambda)=n^{-\frac{k+1}{2}}\,\varSigma_{k}(\lambda)+n^{-\frac{k+1}{2}}\,h(\lambda)$, and the second term verifies
$$\Gel_{n}\left[\left(n^{-\frac{k+1}{2}}\,h(\lambda)\right)^{2}\right]=n^{-(k+1)}\,\,\Gel_{n}\left[h(\lambda)^{2}\right]=O(n^{-1})\to 0.$$
So, $R_{k+1}(\lambda^{*})=n^{-\frac{k+1}{2}}\,\varSigma_{k}(\lambda)+X_{k}$, where $X_{k}$ is a random variable that converges in probability towards zero. Suppose that $k \geq 2$. Then, $\varSigma_{k}$ and $(\varSigma_{k})^{2}$ are observables of Kerov degree $k$ and $2k$, so:
$$\Gel_{n}\left[n^{-\frac{k+1}{2}}\,\varSigma_{k}(\lambda)\right]=O(n^{-\frac{1}{2}})\to 0\qquad;\qquad \Gel_{n}\left[\left(n^{-\frac{k+1}{2}}\,\varSigma_{k}(\lambda)\right)^{2}\right]=O(n^{-1})\to 0 $$
As a consequence, if $k\geq 2$, then $n^{-\frac{k+1}{2}}\,\varSigma_{k}(\lambda)$ converges also in probability towards $0$, so $R_{k+1}(\lambda^{*}) \to 0$.  On the other hand, $R_{2}(\omega)=\varSigma_{1}(\omega)=\frac{p_{2}(\omega)}{2}$ for any continuous diagram $\omega$, so $R_{2}(\lambda^{*})=\frac{\mathcal{A}(\lambda)^{*}}{2}=1$ almost surely under any Gelfand measure. \end{proof}
\bigskip

Now, notice that the LSKV curve has for free cumulants $R_{2}(\Omega)=1$ and $R_{k\geq 3}(\Omega)=0$. Indeed, $\Omega'(s)=\frac{2}{\pi}\,\arcsin\frac{s}{2}$, $\Omega''(s)=\frac{2}{\pi}\,\frac{1}{\sqrt{4-s^{2}}}$, and from this we get
$$p_{k}(\Omega)=\begin{cases} \binom{k}{k/2}&\text{if }k \text{ is even},\\
0&\text{if }k\text{ is odd},
\end{cases}$$
see \cite[Proposition 5.3]{IO02}. It is then a lengthy but straightforward computation to see that $G_{\Omega}(z)=(z-\sqrt{z^{2}-4})/2$, see \emph{e.g.} \cite[\S3]{Mel10b}. As a consequence, $R_{\Omega}(z)=z+1/z$, and this gives us the free cumulants of $\Omega$. So, we have proved that under the Gelfand measures, the free cumulants of scaled Young diagrams converge in probability towards those of $\Omega$.\bigskip

There are various ways to deduce from this the convergence of the shapes; one can for instance\footnote{We want to get rid of the \emph{ad hoc} argument given in \cite[Lemmas 5.6 and 5.7]{IO02} for the Plancherel measures, and that relies on precise estimates of the distribution of the length of a longest increasing subsequence in a random permutation.} use the transition measures of the continuous Young diagrams. If $\omega$ is a continuous Young diagram, its generating function can be written as the Cauchy transform of a unique compactly supported probability measure on $\R$:
$$G_{\omega}(z)=\int_{\R} \frac{\mu_{\omega}(ds)}{z-s}$$
The measure $\mu_{\omega}$ is called the \textbf{transition measure} of $\omega$; its free cumulants are precisely the free cumulants of the Young diagram. In the case of the LSKV curve $\Omega$, $\mu_{\Omega}$ is the unique probability measure with vanishing free cumulants of order $k\geq 3$, and it is known from free probability that this measure is Wigner's semicircle law:
$$\mu_{\Omega}(ds)=\mathbb{1}_{s \in [-2,2]}\,\frac{\sqrt{4-s^{2}}}{2\pi}\,ds$$
The correspondence between continuous Young diagrams $\omega$, generating functions $G_{\omega}(z)$ and probability measures $\mu_{\omega}$ can be extended to homeomorphisms between the three following complete metric spaces:
\begin{enumerate}
\item The space $\mathscr{M}^{1}(\R)$ of probability measures on $\R$ with the topology of convergence in law.
\item The space $\mathscr{N}^{1}$ of analytic functions on the upper half-plane $\mathbb{H}=\{z \,\,|\,\,\Im(z)>0\}$ that satisfy $\Im(N(z))\leq 0$ for all $z\in \mathbb{H}$, and $\lim_{y \to \infty} \I y\,N(\I y)=1$. This space is endowed with a topology slightly stronger than the topology of uniform convergence on compact subsets of $\mathbb{H}$, in order to ensure the completeness of the space.
\item The space $\mathscr{Y}^{1}$ of generalized continuous Young diagrams, that is to say Lipschitz functions with constant $1$ such that
$$\int_{-\infty}^{-1}(1+\omega'(s))\,\frac{ds}{|s|}<\infty\qquad;\qquad \int_{1}^{+\infty}(1-\omega'(s))\,\frac{ds}{|s|}<\infty,$$
two such functions being identified if they differ by a constant. The adequate topology on this space is the topology of uniform convergence on compact subsets. When restricted to continous Young diagrams is the sense of \S\ref{n}, this topology becomes the topology of global uniform convergence.
\end{enumerate}
These homeomorphisms form the so-called \textbf{Markov-Krein correspondence}, see \cite{Ker98} for details. Now, we have seen that $R_{k}(\lambda^{*})$ converges in probability towards $R_{k}(\Omega)$ under the Gelfand measures, and since the free cumulants of measures are related to the usual moments by the combinatorics of noncrossing partitions (\emph{cf.} \cite{NS06}), we have also:
$$\forall k,\,\,\int_{\R}s^{k}\,\mu_{\lambda^{*}}(ds) \to \int_{\R}s^{k}\,\mu_{\Omega}(ds) $$
in probability. But Wigner's semicircle law is characterized by its moments, so in fact one has $\mu_{\lambda^{*}} \to \mu_{\Omega}$ in probability and for Skorohod's topology of convergence in law (\emph{cf.} \cite[Chapter 1]{Bil69} for details on this terminology). By Markov-Krein correspondence, this leads finally to the following result:
\begin{theorem}[First order asymptotics of shapes of Young diagrams under the Gelfand measures]\label{spacebound}
For the topology of uniform convergence, the rescaled shapes $\lambda^{*}$ under the Gelfand measures converge in probability towards the LSKV curve $\Omega$:
$$\forall \eps>0,\,\,\,\Gel_{n}[\|\lambda^{*}-\Omega \|_{\infty}\geq \eps] \to 0.$$
\end{theorem}
\bigskip\bigskip

This theorem raises a legitimate question: is it really surprising to observe the same asymptotic behaviour for Gelfand measures as for Plancherel measures? The answer is no, if one looks at the original proof of the Logan-Shepp-Kerov-Vershik law of large numbers given in \cite{LS77}. Indeed, the dimension of the irreducible representation labelled by a partition $\lambda$ is given by the hook length formula
$$\dim \lambda=\frac{n!}{\prod_{(i,j) \in \lambda} h(i,j)},$$ 
and from this formula, Logan and Shepp constructed a functional $I$ on continuous Young diagrams such that $(\dim \lambda)^{2}$ has roughly the same behaviour as $\E^{-nI(\lambda^{*})}$. Moreover, the curve $\Omega$ is the unique minimizer of the functional $I$, so the Plancherel measure is concentrated on shapes close to $\Omega$, and decreases at exponential speed\footnote{As a consequence, there should be a principle of large deviations for the random partitions under the Plancherel (and possibly Gelfand) measures, but such a principle has never been enounced in a rigourous way.} outside any vicinity of $\Omega$. But of course, if $(\dim \lambda)^{2}\propto \E^{-nI(\lambda^{*})}$, then $(\dim \lambda) \propto \E^{-nI(\lambda^{*})/2}$, so the same argument should hold for Gelfand measures. Thus, one could have given a proof of Theorem \ref{spacebound} relying on this variational principle, but this would not have been sufficient for the study of second order asymptotics.\bigskip\bigskip

\subsection{Second order asymptotics and a central limit theorem for Gelfand measures} Finally, let us prove the analogue of Kerov's central limit theorem for the Gelfand measures. In the following, we denote by $X_{k}$ the random variable $\frac{\varSigma_{k}(\lambda)}{n^{k/2}}$, with by convention $X_{1}=0$; we have seen that $X_{k}$ converges in law towards a gaussian variable $\xi_{k}\sim \mathcal{N}(e_{k},2k)$. As in \cite{IO02} and \cite{Mel10b}, we also denote by $U_{k}(X)$ the $k$-th Chebyshev polynomial of the second kind, with the normalization:
$$U_{k}(2\cos \theta)=\frac{\sin (k+1)\theta}{\sin \theta}$$
The $U_{k}$'s satisfy the recurrence relations $U_{k+2}(X)=X\,U_{k+1}(X)-U_{k}(X)$, with $U_{0}(X)=1$ and $U_{1}(X)=X$; they are the orthogonal polynomials for the semicircle law on $[-2,2]$.
\begin{lemma}[Asymptotic distribution of the moments of the deviation]
We consider the $k$-th moment of the scaled deviation $\frac{\sqrt{n}}{2}(\lambda^{*}(s)-\Omega(s))$, that is to say the random variable 
$$\frac{\sqrt{n}}{2}\int_{\R}s^{k}\,(\lambda^{*}(s)-\Omega(s))\,ds.$$
This is the scaled observable of diagrams $\sqrt{n}\,\frac{p_{k+2}(\lambda^{*})-p_{k+2}(\Omega)}{(k+1)(k+2)}$, and it is equal to
$$\sum_{j=0}^{\lfloor \frac{k-1}{2}\rfloor} \frac{k!}{k+1-j!\,j!}\,X_{k+1-2j},$$
plus a random variable that converges in probability towards $0$ under the Gelfand measures.
\end{lemma}
\begin{proof}
The first part of the lemma is Proposition 7.2 in \cite{IO02}, and it is obvious if one writes $\lambda^{*}(s)-\Omega(s)=(\lambda^{*}(s)-|s|)-(\Omega(s)-|s|)$ in order to make appear the moments of the charges of the continuous diagrams $\lambda^{*}$ and $\Omega$. Then, one uses the expansion of $p_{k+2}$ given at the end of the section \ref{obs}:
$$p_{k+2}(\lambda^{*})=\frac{p_{k+2}(\lambda)}{n^{\frac{k+2}{2}}}=\sum_{j=0}^{\lfloor \frac{k-1}{2}\rfloor} \frac{(k+2)^{\downarrow j+1}}{j!}\,\frac{\varSigma_{k+1-2j}(\lambda)}{n^{\frac{k+2-2j}{2}}}+\begin{cases} \binom{k+2}{(k+2)/2}&\text{if }k\text{ is even},\\0 &\text{if }k\text{ is odd},\end{cases}$$
plus $n^{-\frac{k+2}{2}}$ times an observable of Kerov degree smaller than $k$. By subtracting $p_{k+2}(\Omega)$, one gets rid of the second term of the expansion, and by multiplying by $\frac{\sqrt{n}}{(k+1)(k+2)}$, one obtains:
$$\sqrt{n}\,\,\frac{p_{k+2}(\lambda^{*})-p_{k+2}(\Omega)}{(k+1)(k+2)}=\sum_{j=0}^{\lfloor \frac{k-1}{2}\rfloor} \frac{k!}{k+1-j!\,j!}\,X_{k+1-2j}$$
plus $n^{-\frac{k+1}{2}}$ times an observable of Kerov degree smaller than $k$. Since $n^{\frac{k}{2}}$ is the order of magnitude of such an observable, it is indeed negligible at infinity, whence the result.
\end{proof}\bigskip

\begin{proposition}[Asymptotic distribution of the linear functionals of the Chebyshev polynomials]
For $k \geq 2$, we consider the random variables
$$ \Upsilon_{k}=\frac{\sqrt{n}}{2}\int_{\R} U_{k}(s)\,(\lambda^{*}(s)-\Omega(s))\,ds.$$
As $n$ goes to infinity, the difference $\Upsilon_{k}-\frac{X_{k+1}}{k+1}$ goes to zero in probability. Consequently, the $\Upsilon_{k}$'s converge in law towards independent gaussian variables of laws $\mathcal{N}(\frac{e_{k+1}}{k+1},\frac{2}{k+1})$. 
\end{proposition}
\begin{proof}
As in the proof of Theorem 10 in \cite{Mel10b}, we use the explicit expansion of Chebyshev polynomials:
$$U_{k}(X)=\sum_{m=0}^{\lfloor\frac{k}{2}\rfloor} (-1)^{m}\,\binom{k-m}{m}\,X^{k-2m}$$
Let us denote by $\Theta_{k}$ the scaled observable of the previous lemma; when $n$ goes to infinity, one can replace $\Theta_{k}$ by $\sum_{j=0}^{\lfloor \frac{k-1}{2}\rfloor} \frac{k!}{k+1-j!\,j!}\,X_{k+1-2j}$, and in fact the sum can be taken up to the index $\lfloor \frac{k}{2} \rfloor$, because $X_{1}=0$. Hence:
$$\Upsilon_{k}= \sum_{m=0}^{\lfloor\frac{k}{2}\rfloor} (-1)^{m}\,\binom{k-m}{m}\,\Theta_{k-2m}\qquad;\qquad \Theta_{k} \simeq \frac{1}{k+1}\sum_{j=0}^{\lfloor \frac{k}{2}\rfloor} \binom{k+1}{j}\,X_{k+1-2j}$$
The exact same argument as in \cite[Proposition 7.4]{IO02} and \cite[Theorem 10]{Mel10b} permits then to conclude that $\Upsilon_{k}\simeq \frac{X_{k+1}}{k+1}$, and Theorem \ref{nolove} gives the asymptotics of these variables.\end{proof}\bigskip

Now, any function $f \in \mathscr{C}^{\infty}([-2,2])$ can be expanded in the basis of Chebyshev polynomials, because this is an orthonormal basis of $\mathscr{L}^{2}([-2,2],\mu_{\Omega})$. So, let us write 
$$f(s)=\sum_{k=0}^{\infty} \left(\int_{-2}^{2}f(s)\,U_{k}(s)\, \frac{\sqrt{4-s^{2}}}{2\pi}\,ds\right)\,U_{k}(s).$$
Because of the previous proposition, the linear functional of the function $f$ associated to the scaled deviation $\Delta_{n,\Gel}(s)=\frac{\sqrt{n}}{2}(\lambda^{*}(s)-\Omega(s))$ is asymptotically equal to:
$$\int_{-2}^{2}f(s)\,\left(\sum_{k=1}^{\infty} \frac{\xi_{k+1}}{k+1}\,U_{k}(s)\, \frac{\sqrt{4-s^{2}}}{2\pi}\right)\,ds$$
Hence, as a distribution, $\Delta_{n,\Gel}(s)$ converges weakly in law towards the generalized gaussian process 
$$\Delta_{\infty,\Gel}(s)= \frac{1}{\pi}\sum_{k=1}^{\infty}\frac{\xi_{k+1}}{k+1}\,U_{k}(s)\,\sqrt{4-s^{2}}=\frac{2}{\pi}\sum_{k=2}^{\infty}\frac{\xi_{k}}{k}\,\sin k\theta,\qquad\text{with }s=2\cos\theta,\,\,\theta \in [0,\pi].$$
We write $\xi_{k}=e_{k}+\sqrt{2k}\,\zeta_{k}$, where the $\zeta_{k}$'s are  independent standard gaussian variables. It gives two contributions to $\Delta_{\infty,\Gel}(s)$:
$$\frac{2}{\pi}\sum_{k=1}^{\infty} \frac{\sin(2k+1)\theta}{2k+1}= \frac{1}{2}-\frac{2\sin\theta}{\pi} \qquad\text{and}\qquad \frac{2\sqrt{2}}{\pi}\sum_{k=2}^{\infty}\frac{\zeta_{k}}{\sqrt{k}}\,\sin k\theta = \sqrt{2}\,\Delta_{\infty,\proba}(2\cos \theta).$$
We can finally conclude:
\begin{theorem}[Second order asymptotics of shapes of Young diagrams under the Gelfand measures]
In the sense of distributions on $[-2,2]$, the scaled deviation $\sqrt{n}\,(\lambda^{*}(s)-\Omega(s))$ of the shapes of Young diagrams from the LSKV curve converges in law towards the generalized gaussian process
$$\Delta_{\infty,\Gel}(s)=\Delta_{\infty,\Gel}(2\cos \theta)=\frac{1}{2}-\frac{2\sin\theta}{\pi}+\sqrt{2}\,\Delta_{\infty,\proba}(2\cos \theta)\,$$
where $\Delta_{\infty,\proba}(2\cos \theta)=\frac{2}{\pi}\sum_{k=2}^{\infty}\frac{\zeta_{k}}{\sqrt{k}}\,\sin k\theta$ is the gaussian process already involved in Kerov's central limit theorem.
\end{theorem} 
\noindent The factor $\sqrt{2}$ explains why the fluctuations of random partitions under Gelfand measures seemed slightly bigger than under Plancherel measures; moreover, there is a constant term  $\frac{1}{2} - \frac{2\sin\theta}{\pi}$ in the asymptotic scaled deviation. It is quite a striking result that Kerov's generalized gaussian process $\Delta_{\infty}$ describes again the deviations of the shapes of the random partitions under Gelfand measures (we've already seen that it was the case for Schur-Weyl measures of parameter $\alpha=1/2$); and we could not have guessed or shown that without using our method of noncommutative moments. \bigskip
\bigskip

To conclude with, let us evoke some results due to J. Baik and E. Rains (\emph{cf.} \cite{BR01}). An integer $n$ being fixed, one can construct a one-parameter family of probability measures on $\Part_{n}$ by setting:
$$M_{n,\beta}(\lambda)=\frac{(\dim \lambda)^{\beta}}{\sum_{\mu\in \Part_{n}} (\dim \mu)^{\beta}},\quad\beta>0.$$
These measures are called \textbf{$\mathbf{\beta}$-Plancherel measures}; one recovers the usual Plancherel measure on $\Part_{n}$ when $\beta=2$, and the Gelfand measure when $\beta=1$. An asymptotic expression of the denominator $\sum_{\mu\in \Part_{n}} (\dim \mu)^{\beta}$ is not known in the general case, although one can relate partial sums over partitions with fixed length to some matrix integrals, see \cite[Formula 1.1]{BR01}. Nethertheless, because of the variational argument, it is not difficult to guess that all the $\beta$-Plancherel measures exhibit the same first-order asymptotics (with the LSKV curve as a limit).\bigskip

That said, one can conjecture that at the second order, Kerov's central limit theorem holds for all these measures, up to a multiplicative coefficient (maybe $\sqrt{2/\beta}$), and also up to an additional constant deviation $f_{\beta}(s)$ that is equal to $0$ when $\beta=2$, and to $\frac{1}{2}-\frac{\sqrt{4-s^{2}}}{\pi}$ when $\beta=1$. It is also conjectured that these probability measures on partitions are the good discrete analogues of the beta ensembles coming from random matrix theory (see \emph{e.g.} \cite{DE02}) --- they are random point processes with density
$$d\proba[x_{1},\ldots,x_{n}]=\frac{1}{Z_{n,\beta}}\,|\Delta(x_{1},\ldots,x_{n})|^{\beta}\,\E^{-\frac{1}{2}\sum_{i=1}^{n}(x_{i})^{2}}\,dx_{1}\cdots dx_{n}.$$
In particular, it has been shown in \cite{BR01} that the Baik-Deift-Johansson (\cite{BDJ99}, \cite{BDJ00}, \cite{Oko00}) correspondence relating the asymptotic size of the largest eigenvalues of a random matrix of the GUE ($\beta=2$) to the asymptotic size of the largest parts of a random partition under the Plancherel measures also holds in the case $\beta=1$, that is to say if one replaces the GUE by the GOE and the Plancherel measure by the Gelfand measure. An interesting problem would be to establish analoguous results for generic values of $\beta$. It would also be interesting to relate the second order asymptotics of the $\beta$-Plancherel measures to those of empirical measures of the $\beta$-ensembles, \emph{i.e.}, to their deviations from Wigner's semicircle law.
\bigskip
\bigskip

\bibliographystyle{alpha}
\bibliography{gelfand}

\end{document}